\documentclass[A4paper,12pt]{article}
\usepackage{fullpage}
\usepackage{amssymb, amsmath, amsthm, amsfonts,comment}
\usepackage{mathtools}
\usepackage{thm-restate}
\usepackage{tikz}
\usepackage{subcaption}

\usepackage[hidelinks]{hyperref}
\usepackage[capitalise]{cleveref}

\usepackage[dvipsnames,table]{xcolor}

\usepackage{enumerate}

\newtheorem{thm}{Theorem}[section]
\newtheorem{lem}[thm]{Lemma}
\newtheorem{fact}[thm]{Fact}
\newtheorem{cor}[thm]{Corollary}
\newtheorem{conj}[thm]{Conjecture}
\newtheorem{prop}[thm]{Proposition}

\newtheorem*{thm*}{Theorem}
\newtheorem*{lem*}{Lemma}
\newtheorem*{fact*}{Fact}
\newtheorem*{cor*}{Corollary}
\newtheorem*{conj*}{Conjecture}

\newtheorem{claim}{Claim}[thm]
\newtheorem*{claim*}{Claim}

\theoremstyle{definition}

\newtheorem{define}[thm]{Definition}
\newtheorem{defs}[thm]{Definitions}

\newtheorem*{define*}{Definition}
\newtheorem*{examp*}{Example}

\newtheorem{setup}[thm]{Setup}

\newcommand{\eps}{\varepsilon}

\newcommand*{\claimproofname}{Proof of Claim}
\newenvironment{claimproof}[1][\claimproofname]{\begin{proof}[#1]}{\end{proof}}

\title{A minimum-degree threshold for colour-biased Hamilton cycles in hypergraphs}

\author{Natalie Behague\thanks{School of Mathematical Sciences, Dublin City University, Dublin, Ireland. \\ \indent\indent \hspace{-0.185cm} email: \texttt{ natalie.behague@dcu.ie}. } \and Felix Christian Clemen\thanks{Department of Mathematics and Statistics, University of Victoria, Victoria, B.C., Canada.}$\text{ }^{,}$\thanks{email: \texttt{fclemen@uvic.ca}.} \and Joseph Hyde\thanks{School of Mathematics, Watson Building, University of Birmingham, Edgbaston, Birmingham, B15 2TT. \\ \indent\indent \hspace{-0.185cm}  email: \texttt{ j.f.hyde@bham.ac.uk}.} \and Natasha Morrison\footnotemark[2]$\text{ }^{,}$\thanks{email: \texttt{ nmorrison@uvic.ca}.}}

\begin{document}

\maketitle 
\begin{abstract}
We determine the asymptotically best possible minimum vertex degree condition forcing a two-coloured $3$-graph to contain a colour-biased tight Hamilton cycle. This confirms a conjecture of H\`an, Lang, Marciano, Pavez-Sign\'e, Sanhueza-Matamala, Treglown and Z\'arate-Guer\'en. 
\end{abstract}

\section{Introduction}
In recent years, increasing attention has been given to \textit{colour-bias} problems, which lie within discrepancy theory. Broadly speaking, discrepancy theory examines how far a given system must deviate from perfect uniformity. In this sense, it can be viewed as a counterpart to Ramsey theory: whereas Ramsey theory demonstrates that structured patterns inevitably emerge within disorder, discrepancy theory shows when some amount of imbalance is unavoidable.

Roughly speaking, colour-bias problems concern conditions under which a (hyper)graph $H$ coloured with $r \geq 2$ colours always contains a given substructure $C$ with (significantly) more than $|E(C)|/r$ edges in some colour $c \in [r]$. Note that when all copies of $C$ are (close to) disjoint, there exists a colouring where each copy of $C$ essentially uses all colours equally often. So in order to ensure a colour-bias on some copy of $C$, the conditions on $H$ must guarantee there are many intersecting copies of $C$ in $H$. 

Questions on colour-bias were first investigated in the setting of graphs; see, for example,~\cite{BaloghCsabaJingPluhar2020, BaloghCsabaPluharTreglown2021,Bradac2022,BradaChristophGishboliner2024,FreschiHydeLadaTreglown2021,GishbolinerKrivelevichMichaeli2022b,GishbolinerKrivelevichMichaeli2022}. More recently, attention has shifted to hypergraphs, with a particular focus on conditions for obtaining colour-biased Hamilton cycles and perfect matchings. 
For $k \in \mathbb{N}$, let $H$ be a $k$-uniform hypergraph ($k$-graph) with vertex set $V(H)$.
For $1\leq \ell \leq k$, the \emph{minimum $\ell$-degree} of $H$ is defined to be $\delta_\ell(H) := \min_{S \subseteq V(H); |S| = \ell}\bigl|\{\, e \in E(H) : S \subset e \,\}\bigr|$. 
For $\ell = k-1$, $\delta_{k-1}(H)$ is often referred to as the minimum {\it codegree of $H$} while $\delta_1(H)$ is often referred to as the minimum {\it vertex degree of $H$}.

In this paper, we focus on minimum vertex degree conditions for colour-biased \emph{tight Hamilton cycles}. A tight Hamilton cycle in a $k$-graph $H$ on $n$ vertices is a cyclic ordering 
$v_1, v_2, \ldots, v_n$ of the vertices such that every consecutive $k$-tuple $\{v_i, v_{i+1}, \ldots, v_{i+k-1}\}$
(with indices taken modulo $n$) forms an edge of $H$.

Before turning to the coloured setting, we briefly discuss some relevant extremal thresholds. Building on a number of earlier results (see \cite{GlebovPersonWeps2012,RodlRucinski2014,RodlRucinskiSchachtSzemeredi2017}), Reiher, R\"odl, Ruci\'nski, Schacht, and Szemer\'edi~\cite{3hamcycle} gave an asymptotically best possible condition for a $3$-graph to contain a tight Hamilton cycle. 
\begin{thm}[Reiher, R\"odl, Ruci\'nski, Schacht, and Szemer\'edi~\cite{3hamcycle}]\label{thm:hamcyclenocolour}
For every $\alpha>0$ there exists an integer $n_0$ such that every $3$-graph $H$ on $n\geq n_0$ vertices with minimum vertex degree $\delta_1(H) \ge \left(\frac{5}{9} + \alpha\right) \binom{n}{2}$
contains a tight Hamilton cycle. 
\end{thm}

The corresponding perfect matching threshold was first studied by H\`an, Person, and Schacht~\cite{HPS09}, who obtained asymptotic bounds, and later determined precisely by K\"uhn, Osthus, and Treglown~\cite{KOT10}. Theorem~\ref{thm:hamcyclenocolour} shows that, asymptotically, the minimum vertex-degree thresholds for perfect matchings and tight Hamilton cycles in $3$-graphs coincide. One might expect a similar phenomenon for 4-graphs. However, the thresholds do not coincide there, as shown by a construction outlined in~\cite[Section 5.2]{color-bias25}. 

Returning to the coloured setting, minimum vertex-degree conditions for colour-biased perfect matchings have attracted attention in the past few years~\cite{BTZ24, color-bias25, LMX26}. Most recently, asymptotically tight minimum vertex degree conditions for colour-biased perfect matchings have been established by H\`an,  Lang, Marciano, Pavez-Sign\'e,  Sanhueza-Matamala, Treglown and Z\'arate-Guer\'en~\cite{color-bias25}. They also
conjectured~\cite[Conjecture 5.1]{color-bias25} that in 3-graphs edge-coloured with 2 colours, the two thresholds should coincide, at least asymptotically. We prove this conjecture.

\begin{restatable}{thm}{main}
\label{thm:main}
    For all $\alpha > 0$ there exist $\delta, n_0 > 0$ such that the following holds. Let $H$ be a red/blue coloured 3-graph on $n \geq n_0$ vertices with $\delta_1(H) \geq (\frac{3}{4} + \alpha)\binom{n}{2}$. Then $H$ contains a tight Hamilton cycle with at least $(\frac{1}{2} + \delta)n$ edges of the same colour.
\end{restatable}
The degree condition in Theorem~\ref{thm:main} is asymptotically best possible. To see this, let $V = V_1 \sqcup V_2$ be a partition of the vertex set with 
$||V_1|-|V_2||\leq 1$. 
Define a $3$-graph $H$ whose edge set consists of all triples that intersect both $V_1$ and $V_2$. 
Colour an edge red if it contains exactly two vertices in $V_1$, and blue if it contains exactly one vertex in $V_1$ (see Figure~\ref{fig:extremal}). 
This construction satisfies $\delta_1(H)\geq \binom{n-1}{2} - \max_{i=1,2} \binom{|V_i|-1}{2} \ge  \frac{3}{4}\binom{n-1}{2} +\frac{n-1}{8}$, but every tight Hamilton cycle yields a difference of at most a constant between its numbers of red and blue edges. 
\begin{figure}[h!]
\begin{center}
\tikzset{
vtx/.style={inner sep=1.1pt, outer sep=0pt, circle, fill,draw}, 
hyperedge/.style={fill=pink,opacity=0.5,draw=black}, 
vtxBig/.style={inner sep=12pt, outer sep=0pt, circle, fill=white,draw}, 
hyperedge/.style={fill=red,opacity=0.5,draw=black}, 
hyperedge2/.style={fill=blue,opacity=0.5,draw=black}, 
}
\begin{tikzpicture}[scale=2.0]
\draw (0,0) coordinate(x1) ellipse (0.25cm and 0.7cm);
\draw (1,0) coordinate(x1) ellipse (0.25cm and 0.7cm);
\draw
(0,-0.2) coordinate(1) 
(0,-0.4) coordinate(2) 
(0,0.3) coordinate(3) 
(1,0.2) coordinate(4) 
(1,-0.3) coordinate(5) 
(1,0.4) coordinate(6) 
;
\draw[hyperedge] (1) to[bend left] (2) to[bend left=5] (5) to[bend left=5] (1);
\draw[hyperedge2] (4) to[bend left] (6) to[bend left=5] (3) to[bend left=5] (4);
\draw (0,0) node{$V_1$};
\draw (1,-0) node{$V_2$};
\end{tikzpicture}
\end{center}
\caption{Example showing the minimum vertex degree condition in Theorem~\ref{thm:main} is asymptotically best possible.}\label{fig:extremal}
\end{figure}
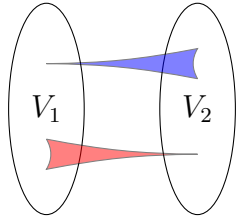

The proof of Theorem~\ref{thm:main} is based on the concept of \textit{switchers}, pairs of collections of tight paths on the same vertex set whose red-blue edge counts differ.  If there are few switchers, then our key lemma, Lemma~\ref{lem:key}, yields a strong global characterisation of the colouring. If there are many switchers, we build a cycle containing them, and utilise their `switching' property to obtain the required colour-bias. In the minimum vertex degree setting, many pairs of vertices may fail to lie in any edge, and consequently our switchers cannot be connected by simple local extension arguments. Given this, we can only utilise switchers that satisfy a global connectivity property. This means, in contrast to previous literature on colour-bias problems, we require additional technical arguments to overcome these barriers.



Working with codegree conditions is more straightforward, as if every pair of vertices is contained in many edges it is easy to extend any tight path by another edge. Indeed, analogous results exist when one considers codegree, as opposed to minimum vertex degree. Results establishing codegree conditions that guarantee tight Hamilton cycles were obtained in~\cite{RodlRucinskiSzemeredi2006,RodlRucinskiSzemeredi2011}. More recently, Gishboliner, Glock, and Sgueglia~\cite{gishboliner2025tight} determined asymptotically optimal minimum codegree (as opposed to minimum vertex degree) conditions for colour-biased tight Hamilton cycles in $r$-edge-coloured $k$-uniform hypergraphs for all $r \ge 2$ and $k \ge 3$. 

The paper is organised as follows. In Section~\ref{sec:Not} we introduce the notation used throughout. In Section~\ref{sec:proofoutline} we provide a more detailed proof outline, expanding upon the brief sketch given above. In Section~\ref{sec:manys} we treat the case where there are many switchers, while in Section~\ref{sec:key_lemma} we consider the case with no switchers. In Section~\ref{sec:fews} we complete the proof by handling the regime with few switchers. Finally, in Section~\ref{sec:concl} we conclude with some natural directions for future research.

\section{Definitions and proof outline}

\subsection{Notation and basic definitions}
\label{sec:Not}

A \emph{tight path} in a $k$-graph $H$ on $t$ vertices is an ordering 
$v_1, v_2, \ldots, v_t$ of the vertices such that every consecutive $k$-tuple $\{v_i, v_{i+1}, \ldots, v_{i+k-1}\}$
 forms an edge of $H$. We say that a (tight) path has \emph{length $\ell$} for some $\ell \in \mathbb{N}$ if the path is on $\ell$ edges. For a 2-graph $G$ and distinct vertices $x,y \in V(G)$, an \emph{$x$-$y$-path} is a path with endpoints $x$ and $y$. Similarly, for a $3$-graph $H$ and distinct pairs of vertices $(x,y)$ and $(z,w)$ - that is, $x \neq y \neq z \neq w$ - an \emph{$(x, y)$---$(z, w)$-path} is a tight path beginning with vertices $x$ and $y$ and ending with vertices $z$ and $w$. That is, $x$ and $w$ are the only vertices of vertex degree one in the tight path and $y$ and $z$ are the only vertices of vertex degree two in the tight path. For a set of vertices $U$ in $G$, we define $N(U) := \{v \in V(G)\setminus U: \exists u \in U \text{ s.t. } vu \in E(G)\}$.

It will be useful to associate a value to each colour so that we can find the difference in the number of edges of each colour just by summing the values.
\begin{define}
For an edge $e$ in $H$ let
\[
c(e) =
\begin{cases}
    +1 & \text{if $e$ is red}, \\
    -1 & \text{if $e$ is blue},
\end{cases}
\]
and for any set of edges $E$, let 
\[c(E) \coloneqq \sum_{e \in E} c(e)\]
which we call the \emph{edge colour sum of $E$}.
\end{define}
Observe that a  tight Hamilton cycle $C$ has at least $(\frac{1}{2} + \delta)n$ edges of the same colour if and only if $|c(C)| \geq 2\delta n$.

\begin{define}
    Let $H = (V,E)$ be a 3-graph and $v \in V$. We define the \emph{link graph $L_v$ of $v$} to be the 2-graph with $V(L_v) := V$ and $E(L_v) := \{yz : vyz \in E(H)\}$.
\end{define} 

Finally, we use the standard hierarchy notation $a \ll b$ to denote that there exists a non-decreasing function $f: (0, 1] \to (0,1]$ such that $f(a) \leq b$.

\subsection{Proof outline}\label{sec:proofoutline}

In this section, we give a more detailed outline of the proof of Theorem~\ref{thm:main}.

Recall from the previous section that the proof utilises `switchers', which consist of pairs of collections of tight paths on the same vertex set with differing colour counts. It is vital for us to be able to connect and extend these via tight paths. Having no minimum codegree condition means that we cannot hope to extend tight paths in a greedy way, that is, appending one edge at a time. We overcome this limitation by exploiting our minimum vertex degree condition to find large highly connected components of the link graphs of each vertex. Since these `robust' components are large, they overlap significantly, providing us with pairs of vertices contained in many of these components. These `well-connected' pairs of vertices will help us to extend and join together tight paths as well as switchers.

\begin{define}[$(\beta,\ell)$-robust {\cite[Definition 2.2]{3hamcycle}}]
       Let $\beta>0$ and $\ell$ be a positive integer. A 2-graph $R$ is said to be $(\beta,\ell)$-robust if for any two distinct vertices $x,y \in V(R)$ the number of $x$-$y$-paths in $R$ of length $\ell$ is at least $\beta|V(R)|^{\ell-1}$.
\end{define}
The following proposition asserts that each link graph contains a robust subgraph with many vertices and edges. Since its proof is similar to that of \cite[Prop 2.3]{3hamcycle}, we defer it to Appendix~\ref{app:robust}, where we prove a more general statement. 

\begin{restatable}[cf.~{\cite[Prop.~2.3]{3hamcycle}}]{prop}{robust}
\label{prop:robust} Let $\alpha > 0$ and let $H$ be a 3-graph on $n$ vertices, where $n$ is sufficiently large, satisfying $\delta_1(H) \geq (\frac{3}{4} + \alpha)\binom{n}{2}$. There exist $\beta>0$ and an odd integer $\ell\geq 3$ such that the following holds. For every $v \in V(H)$, there exists an induced subgraph $R_v\subseteq L_v$ satisfying 
\begin{itemize}
\item[(i)] $|V(R_v)|\geq \left(\frac{2+\sqrt{2}}{4} +\frac{\alpha}{2}\right)n \ge \left(0.853 +\frac{\alpha}{2}\right)n$,
\item[(ii)] $e_L(V(R_v),V\setminus V(R_v))\leq \alpha n^2/4$ and $e(R_v)\geq \left(\frac{3}{4}+\frac{\alpha}{2}\right)\frac{n^2}{2}-\frac{(n-|V(R_v)|)^2}{2}$,
\item[(iii)] $R_v$ is $(\beta,\ell)$-robust. 
\end{itemize}
\end{restatable}
We may refer to $R_v$ as a \emph{robust} subgraph. Given~\cref{prop:robust}, we introduce the following `setup' for brevity.
\begin{setup} \label{setup} In what follows, 
fix $\alpha > 0$ and take $n_0 = n_0(\alpha)$ sufficiently large. Let $H$ be a red/blue edge-coloured 3-graph  on $n \ge n_0$ vertices with $\delta_1(H) \geq (\frac{3}{4} + \alpha)\binom{n}{2}$, as in the statement of Theorem~\ref{thm:main}. Further, take $\beta > 0$, an odd integer $\ell$ and graphs $R_v \subseteq L_v$ for $v \in V(H)$ as given by Proposition~\ref{prop:robust} with inputs $\alpha$ and $H$. Moreover, define constants $\gamma, \zeta$ and $\delta$ using the following hierarchy, choosing from right to left:\begin{equation}\label{eq:hierarchy} 
\frac{1}{n_0} \ll \delta \ll \zeta \ll \beta, \frac{1}{\ell}, \gamma \ll \alpha.
\end{equation}

\end{setup}

Let us now continue with the following important notion for connecting up tight paths into a tight Hamilton cycle, which formalises our intuitive idea of `well-connected' pairs. 

\begin{define}[$\rho$-connectable]\label{def:connectable}
    For a constant $\rho > 0$, a pair $xy$ of vertices in $V$ is said to be $\rho$-connectable if the set $$U_{xy} = \{v \in V: xy \in E(R_v)\}$$ of all vertices $v$ having $xy$ as an edge of their robust subgraph has size $|U_{xy}| \geq \rho|V|$.
\end{define}

We now formally define the notion of a \emph{switcher}.

\begin{defs}[Switchers]
A \emph{potential switcher} in a red-blue coloured $3$-graph $H$ is a pair $(\mathcal{S}, \mathcal{S}')$ of collections of vertex disjoint tight paths $\mathcal{S} = \{P_1,\ldots,P_t\}$ and $\mathcal{S}' = \{P'_1,\ldots,P'_t\}$  (for some $t \in \mathbb{N}$) such that $V(\bigcup \mathcal{S}) = V(\bigcup \mathcal{S}')$, and such that both $P_i$ and  $P_i'$ have the same initial and final pairs of vertices for all $1 \le i \le t$.

A potential switcher is a \emph{switcher} (also called a \emph{colour-changing switcher}) if the edge colour sums of these collections of tight paths are not equal, that is $c(\mathcal{S}) \ne c(\mathcal{S}')$.

For a constant $\zeta > 0$, a (potential) switcher is called \emph{$\zeta$-connectable} if the initial and final pairs of all paths (in $\mathcal{S} \cup \mathcal{S}'$) are $\zeta$-connectable, and is called \emph{$h$-bounded} if $|V(\bigcup \mathcal{S})| \leq h$, for an integer $h > 0$. 

We say a set of switchers is \emph{disjoint} if their vertex sets are pairwise disjoint. 
\end{defs}

\bigskip

For the proof of Theorem~\ref{thm:main}, given in Section~\ref{sec:fews}, we conclude differently depending on whether there are `few' or `many' $\zeta$-connectable switchers. When there are few switchers, we delete them and apply our Key Lemma (Lemma~\ref{lem:key} below) to the resulting switcher-free 3-graph. In order to state this, we need the following definition.

\begin{define}[Robust edges]
    An edge $xyz$ in $H$ is said to be \emph{robust} if at least one of the following holds: $$xy \in E(R_z), \ xz \in E(R_y), \ yz \in E(R_x).$$ 
\end{define}

\begin{restatable}[Key Lemma]{lem}{key}
\label{lem:key}
   Given Setup \ref{setup}, suppose that $H$ contains no $\zeta$-connectable $7\ell$-bounded switchers. Then there exists a bipartition of the vertex set of $H$ into (possibly empty) parts $X$ and $Y$ such that one of the following cases holds. 
    \begin{enumerate}[{Case} A.]
        \item For all but at most $\gamma n^3$ robust edges $e$, we have $|e\cap Y| \in \{0,1\}$ such that
        \begin{itemize}
            \item if $|e\cap Y| = 0$ then $e$ is red,
            \item if  $|e\cap Y| = 1$ then $e$ is blue.
        \end{itemize}
        \item For all but at most $\gamma n^3$ robust edges $e$, we have $|e\cap Y| \in \{1,2\}$ such that
        \begin{itemize}
            \item if $|e\cap Y| = 1$ then $e$ is red,
            \item if  $|e\cap Y| = 2$ then $e$ is blue.
        \end{itemize}
       \item For all but at most $\gamma n^3$ robust edges $e$, we have $|e\cap Y| \in \{2,3\}$ such that
        \begin{itemize}
            \item if $|e\cap Y| = 2$ then $e$ is red,
            \item if  $|e\cap Y| = 3$ then $e$ is blue.
        \end{itemize}
    \end{enumerate}
\end{restatable}
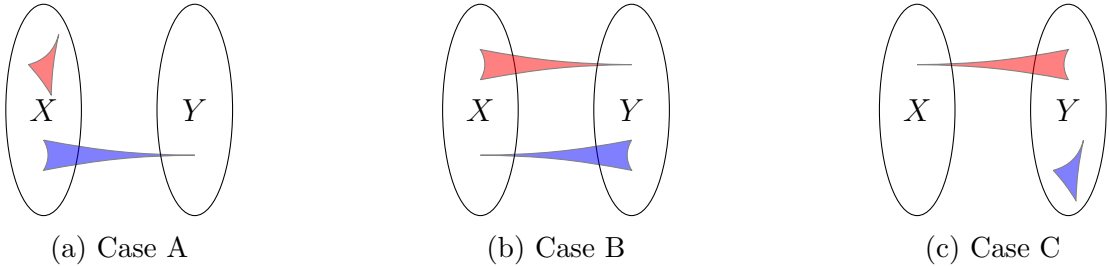
\begin{figure}[h!]
\centering

\tikzset{
vtx/.style={inner sep=1.1pt, outer sep=0pt, circle, fill,draw}, 
vtxBig/.style={inner sep=12pt, outer sep=0pt, circle, fill=white,draw}, 
hyperedge/.style={fill=red,opacity=0.5,draw=black}, 
hyperedge2/.style={fill=blue,opacity=0.5,draw=black}
}

\begin{subfigure}{0.3\textwidth}
\centering
\begin{tikzpicture}[scale=2.0]
\draw (0,0) ellipse (0.25cm and 0.7cm);
\draw (1,0) ellipse (0.25cm and 0.7cm);
\draw
(0,-0.2) coordinate(1)
(0,-0.4) coordinate(2)
(0,0.3) coordinate(3)
(1,0.2) coordinate(4)
(1,-0.3) coordinate(5)
(1,0.4) coordinate(6);
\draw[hyperedge2] (1) to[bend left] (2) to[bend left=5] (5) to[bend left=5] (1);
\draw[hyperedge] (-0.1,0.3) to[bend right] (0.1,0.5) to[bend right=5] (0.05,0.1) to[bend right=5] (-0.1,0.3);
\draw (0,0) node{$X$};
\draw (1,0) node{$Y$};
\end{tikzpicture}
\caption{Case A}
\end{subfigure}
\hfill
\begin{subfigure}{0.3\textwidth}
\centering
\begin{tikzpicture}[scale=2.0]
\draw (0,0) ellipse (0.25cm and 0.7cm);
\draw (1,0) ellipse (0.25cm and 0.7cm);
\draw
(1,-0.2) coordinate(1)
(1,-0.4) coordinate(2)
(1,0.3) coordinate(3)
(0,0.2) coordinate(4)
(0,-0.3) coordinate(5)
(0,0.4) coordinate(6);
\draw[hyperedge2] (1) to[bend right] (2) to[bend right=5] (5) to[bend right=5] (1);
\draw[hyperedge] (4) to[bend right] (6) to[bend right=5] (3) to[bend right=5] (4);
\draw (0,0) node{$X$};
\draw (1,0) node{$Y$};
\end{tikzpicture}
\caption{Case B}
\end{subfigure}
\hfill
\begin{subfigure}{0.3\textwidth}
\centering
\begin{tikzpicture}[scale=2.0]
\draw (0,0) ellipse (0.25cm and 0.7cm);
\draw (1,0) ellipse (0.25cm and 0.7cm);
\draw
(0,-0.2) coordinate(1)
(0,-0.4) coordinate(2)
(0,0.3) coordinate(3)
(1,0.2) coordinate(4)
(1,-0.3) coordinate(5)
(1,0.4) coordinate(6);
\draw[hyperedge2] (0.9,-0.4) to[bend right] (1.1,-0.2) to[bend right=5] (1.05,-0.6) to[bend right=5] (0.9,-0.4);
\draw[hyperedge] (4) to[bend left] (6) to[bend left=5] (3) to[bend left=5] (4);
\draw (0,0) node{$X$};
\draw (1,0) node{$Y$};
\end{tikzpicture}
\caption{Case C}
\end{subfigure}

\caption{Illustration of the majority of the robust edges in the three cases from Lemma~\ref{lem:key}.}

\end{figure}

Lemma~\ref{lem:key} is proved in Section \ref{sec:key_lemma}, and two specific types of switcher are used in its proof. The first is a `short' switcher where each of $\mathcal{S}$ and $\mathcal{S}'$ contains two tight 3-paths on three edges. Precisely, $\mathcal{S} = \{abucd, wxvyz\}$ and $\mathcal{S}' = \{abvcd, wxuyz\}$ for distinct vertices $a,b,c,d,u,v,w,x,y,z$.  See Figure~\ref{fig:short_switcher} for a diagram. The absence of this short switcher will be used to show a local property of the edge colouring. That is, for almost all vertices $u,v \in V(H)$, the quantity $c(uxy) - c(vxy)$ is constant on almost all edges $xy \in R_u \cap R_v$.

\begin{figure}[h]
    \centering
    \includegraphics[scale=1]{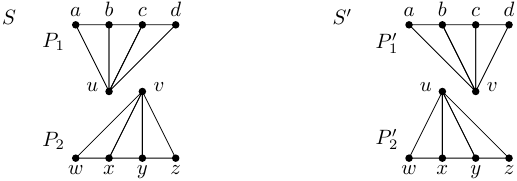}
    \caption{A `short' switcher where each of $S$ and $S'$ contains two tight 3-paths on three edges.}
    \label{fig:short_switcher}
\end{figure}

The second type of switcher is a `long' switcher, where each of $\mathcal{S}$ and $\mathcal{S}'$ contains one long path on $\tfrac{3}{2}(\ell+3) $ edges and $\ell$ short paths on three edges (see Figure~\ref{fig:long_switcher}). We  show that if there are none of these long switchers then the local property of the edge colourings is in fact globally consistent. 
 
\begin{figure}[h!]
    \centering
    \includegraphics[scale=1]{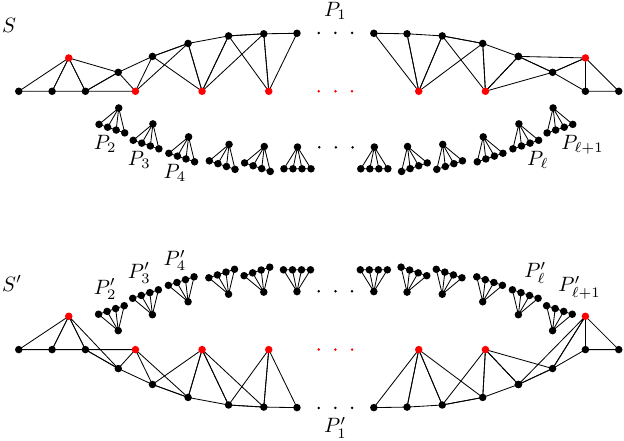}
    \caption{A `long' switcher. There are $\tfrac{1}{2}(\ell+3)$ highlighted (in red) vertices on each of the long paths $P_1$ and $P_1'$, which each contain $\tfrac{3}{2}(\ell+3)$ edges. Note that, although the figure does not show it, for $2 \le i \le \ell+1$, corresponding short paths $P_i,P_i'$ only differ in the central vertex.}
    \label{fig:long_switcher}
\end{figure}

When there are many switchers, we will conclude via the following lemma, proved in Section~\ref{sec:manys}.

\begin{restatable}{lem}{lotsofswitchers}
\label{lem:lots_of_switchers}
     Given Setup \ref{setup}, if $H$ contains at least $4\delta n$ disjoint $\zeta$-connectable $7\ell$-bounded switchers then $H$ contains a tight Hamilton cycle with at least $(\frac{1}{2} + \delta)n$ edges of the same colour.
\end{restatable}

The proof of Lemma~\ref{lem:lots_of_switchers} is similar to the proof of Theorem~\ref{thm:hamcyclenocolour} from \cite{3hamcycle}. The main difference is that we create a Hamilton cycle that `contains' the disjoint switchers, in the sense that it contains one collection of paths from each switcher; this is made possible by requiring the switchers to be $\zeta$-connectable. We can then `switch' between the collections of paths in each switcher to get a colour-biased Hamilton cycle.

\section{Many switchers: proof of Lemma~\ref{lem:lots_of_switchers}}
\label{sec:manys}

Say that a tight path or cycle \emph{contains} a switcher $(\mathcal{S}, \mathcal{S}')$ if it contains the tight paths in $\mathcal{S}$. Using Lemma~\ref{lem:connecting} (below), we will glue the $4\delta n$ disjoint $7\ell$-bounded $\zeta$-connectable switchers together to find a tight path $\mathcal{W}$ containing them all. 

First, we set aside a `reservoir' $\mathcal{R}$, a small set of vertices that can be used to link together any endpoints of paths (\cref{lem:reservoir}). We then construct an `absorbing path' $P_A$ which can be extended to a path containing any small set of unused leftover vertices (\cref{lem:absorb}). 
For the final ingredient, we find an almost spanning path $Q$ disjoint from $\mathcal{W}$, $P_A$ and most of $\mathcal{R}$ (\cref{lem:bigpath}).

Putting this all together, we use the reservoir $\mathcal{R}$ to glue together the paths $\mathcal{W}, P_A$ and $Q$ into an almost spanning cycle. We then extend the absorbing path $P_A$ to cover the small number of unused vertices, thus obtaining a Hamilton cycle $C$ containing all of the switchers. For each switcher, we have the choice to switch to using $\mathcal{S}'$ instead of $\mathcal{S}$, giving a colour-biased tight Hamilton cycle as required. 

We begin by laying out results from \cite{3hamcycle} we will need. Throughout the section we assume Setup \ref{setup} holds.

\begin{lem}[Connecting Lemma, {\cite[Proposition~2.6]{3hamcycle}}]
   \label{lem:connecting} Let $\zeta > 0$. Then there exists $\vartheta > 0$
    such that every two disjoint $\zeta$-connectable ordered pairs $(x, y)$ and $(z, w)$ are connected by
    at least $\vartheta n^{3\ell + 1}$ tight $(x, y)$---$(z, w)$-paths of length $3(\ell+1)$ in $H$.
\end{lem}

\cref{lem:connecting} gives us a new constant $\vartheta$, which we may assume satisfies $\vartheta \ll \zeta$. In addition, we introduce new constants $\zeta_*, \vartheta_*, \zeta_{**}$ and $\vartheta_{**}$, whose scope is only this section and which are necessary to use certain lemmas from~\cite{3hamcycle} as black boxes. These augment the hierarchy in \cref{setup} as follows:

\[ \frac{1}{n_0} \ll \delta \ll \vartheta_{**} \ll \zeta_{**} \ll \vartheta_* \ll \zeta_* \ll \vartheta \ll  \zeta \ll \beta, \frac{1}{\ell} \ll \alpha \]

These constants are obtained in~\cite{3hamcycle} by applying \cref{lem:connecting} more times, first with some $\zeta_* \ll \vartheta$ and then with some $\zeta_{**} \ll \vartheta_*$. We keep the same symbol names to ease the use of the lemmas from that paper as black boxes.

We can now state the reservoir and absorbing lemmas that are central to the proof of Lemma~\ref{lem:lots_of_switchers}. 

\begin{lem}[Reservoir Lemma, {\cite[Proposition~$2.7$]{3hamcycle}}]\label{lem:reservoir}
    There exists a `reservoir' set $\mathcal{R} \subseteq V(H)$ with $\frac{\vartheta_{*}^2 n}{2} \leq |\mathcal{R}| \leq \vartheta_{*}^2 n$ such that for all disjoint pairs of $\zeta_{**}$-connectable pairs $(x,y)$ and $(z,w)$ there are at least $\vartheta_{**}|\mathcal{R}|^{3\ell + 1}/2$ tight $(x,y)$---$(z,w)$-paths of length $3(\ell + 1)$ in $H$ whose internal vertices belong to $\mathcal{R}$.
\end{lem}

We will need to find an absorbing path avoiding both this `reservoir' set $\mathcal{R}$, given by Lemma~\ref{lem:reservoir}, and a path $\mathcal{W}$ containing the switchers. In \cite{3hamcycle}, only the reservoir set need be avoided, but inspection of the proof of \cite[Proposition~2.9]{3hamcycle} reveals that the same argument works if we instead avoid some specified set of size $2\vartheta^2_{*} n$. We will see later that $|V(\mathcal{W})| \leq \vartheta^2_{*} n$, so this will suffice for our purposes.

\begin{lem}[Absorbing path, cf. {\cite[Proposition~$2.9$]{3hamcycle}}]\label{lem:absorb}
    Let $\mathcal{R}' \subseteq V(H)$ be a set of at most $2\vartheta^2_{*} n$ vertices. There exists a tight path $P_A$ which is a subhypergraph of $H - \mathcal{R}'$ and has the following properties:
    \begin{itemize}
        \item[(i)] $|V(P_A)| \leq \vartheta_{*}n$,
        \item[(ii)] the end-pairs of $P_A$ are $\zeta_{*}$-connectable, and
        \item[(iii)] for every set $X \subseteq V \setminus V(P_A)$ with $|X| \leq 2\vartheta_{*}^2n$ there is a tight path in $H$ whose set of vertices is $V(P_A) \cup X$ and whose end-pairs are the same as those of $P_A$.
    \end{itemize}
\end{lem}

The final ingredient needed for the proof of Lemma~\ref{lem:lots_of_switchers} is the following lemma. As with Lemma~\ref{lem:absorb}, this is an analogue of a result from \cite{3hamcycle} where we additionally remove a set of vertices $\mathcal{W}$ containing our $4\delta n$ switchers (this time without also removing the reservoir $\mathcal{R}$). As above, inspection of the proof of \cite[Lemma~7.1]{3hamcycle} reveals that the same argument works if we instead avoid a small enough linearly sized set.

\begin{lem}[{\cite[Lemma~$7.1$]{3hamcycle}}]\label{lem:bigpath}
    Let $\mathcal{R} \subseteq V$ be a reservoir set as given by Lemma~\ref{lem:reservoir}, let $\mathcal{W} \subseteq V$ be a set of at most $100\delta \ell^2 n$ vertices and let $P_A$ be an absorbing path given by Lemma~\ref{lem:absorb}. There exists a tight path $Q \subseteq \hat{H} = H - V(P_A) - V(\mathcal{W})$ such that: 

    \begin{itemize}
        \item[(i)] $|V(\hat{H})\setminus (\mathcal{R} \cup V(Q))| \leq \vartheta_{*}^2 n$. 
        \item[(ii)] $|V(Q) \cap \mathcal{R}| \leq \vartheta_{**}^2 n$.
        \item[(iii)] The end pairs of $Q$ are $\zeta_{**}$-connectable.
    \end{itemize}
\end{lem}
Given Lemmas~\ref{lem:reservoir}, \ref{lem:absorb}, and \ref{lem:bigpath}, we are able to 
prove Lemma~\ref{lem:lots_of_switchers}, restated below for convenience.

\lotsofswitchers*

\begin{proof}[Proof]
     Let the disjoint $7\ell$-bounded switchers be denoted by $W_1, \ldots, W_{\lceil 4 \delta n \rceil}$.
     By definition, for each $1 \le i \le \lceil 4 \delta n \rceil$, 
     there exists $t_i \leq 7\ell$ such that each $W_i$ consists of two sets of paths $\mathcal{S}_i = \{P_{i,1}, \ldots, P_{i,t_i}\}$ and $\mathcal{S}_i' = \{P_{i,1}', \ldots, P_{i,t_i}'\}$. In addition, $P_{i,j}$ and $P_{i,j}'$ have the same initial pairs of vertices $(x^1_{i,j}, x^2_{i,j})$ and the same final pairs of vertices $(y^1_{i,j}, y^2_{i,j})$ for each $j$. 
    Observe that since $c(\mathcal{S}_i) \neq c(\mathcal{S}_i')$ for all  $1 \le i  \le \lceil 4 \delta n \rceil$, 
     we may without loss of generality relabel for each $i$ such that $c(\mathcal{S}_i) \geq c(\mathcal{S}_i') + 1$. 

     We will begin by repeatedly applying Lemma~\ref{lem:connecting} to join the paths $P_{i,j}$ together. 
     For each $1 \le i \le \lceil 4 \delta n \rceil$ and  $1 \le s \le t_i-1$, apply Lemma~\ref{lem:connecting} to create vertex disjoint tight paths of length $3(\ell + 1)$ in $H$ between $(y^1_{i,s}, y^2_{i,s})$ and $(x^1_{i,s+1}, x^2_{i,s+1})$.
     Further, for each $1 \le i \le \lceil 4 \delta n \rceil - 1$, apply Lemma~\ref{lem:connecting} to create vertex disjoint tight paths of length $3(\ell + 1)$ in $H$ between $(y^1_{i,t_i}, y^2_{i,t_i})$ and $(x^1_{i+1,1}, x^2_{i+1,1})$.
     
     Note that with each application of Lemma~\ref{lem:connecting}, we make sure the internal vertices of the tight paths added are vertex disjoint from both each other and $\bigcup_{i =1}^{\lceil4\delta n\rceil} V(W_i)$.  This is possible, as we apply the lemma
     at most $\sum_{i=1}^{\lceil 4 \delta n \rceil} t_i \le \lceil 4 \delta n \rceil \cdot 7\ell$ times, adding at most $\lceil 4 \delta n \rceil  \cdot 7\ell \cdot (3\ell +1)$ vertices. Thus we have at most $\lceil 4 \delta n \rceil  \cdot 7\ell \cdot (3\ell +2) < 100\delta \ell^2 n$ vertices to avoid (those in the switchers and those added in previous paths). For any $x,y,w,z$, each fixed vertex to avoid can be in at most $n^{3\ell}$ tight $(x,y)$---$(z,w)$-paths on $3(\ell+1)$ edges, so there are at most $100\delta \ell^2 n^{3\ell+1}$ paths ruled out at each step. Since there are at least $\vartheta n^{3\ell+1}$ tight $(x,y)$---$(z,w)$-paths to choose from and $\vartheta \gg \delta \ell^2$, we can therefore always find a path that is internally disjoint from the previously used vertices.
          
     Denote by $\mathcal{W}$ the tight path we have thus created and let $\mathcal{W}'$ be the tight path on the same vertex set obtained by replacing the paths in $\mathcal{S}_i$ by the paths in $\mathcal{S}_i'$. Note that $\mathcal{W}$ and $\mathcal{W}'$ have the same initial and final pairs
     $(w_1, w_2) := (x^1_{1,1}, x^2_{1,1})$ and $(w'_1, w'_2) := (y^1_{ \lceil 4 \delta n \rceil, t_{ \lceil 4 \delta n \rceil}}, y^2_{ \lceil 4 \delta n \rceil, t_{ \lceil 4 \delta n \rceil}})$. Moreover, since $c(\mathcal{S}_i) \ge c(\mathcal{S}_i') + 1$, we have $c(\mathcal{W}) \geq c(\mathcal{W}') + 4\delta n$. Note further that $|V(\mathcal{W})| \leq 100\ell^2 \delta n$. 
     
    Now apply Lemmas~\ref{lem:reservoir}, \ref{lem:absorb} and \ref{lem:bigpath} to $H$, yielding a reservoir set $\mathcal{R}$, an absorbing path $P_A \subseteq H - \mathcal{R} - V(\mathcal{W})$ and a tight path $Q \subseteq \hat{H} := H  - V(P_A) - V(\mathcal{W})$ such that: 
 
     \begin{itemize}
         \item[(i)] $|V(\hat{H})\setminus (\mathcal{R} \cup V(Q))| \leq \vartheta_{*}^2 n$.
         \item[(ii)] $|V(Q) \cap \mathcal{R}| \leq \vartheta_{**}^2 n$.
         \item[(iii)] The end pairs of $Q$, denoted $(q_1,q_2)$ and $(q_1',q_2')$, are $\zeta_{**}$-connectable.
     \end{itemize}
     
     By Lemma~\ref{lem:reservoir}, every two disjoint $\zeta_{**}$-connectable pairs in $H$ are connected by at least $\vartheta_{**}|\mathcal{R}|^{3\ell + 1}/2$ tight paths of length $3(\ell + 1)$ in $H$ whose internal vertices belong to $\mathcal{R}$.

     Our goal now is to (one by one) join the end pairs of $\mathcal{W}$, $P_A$ and $Q$ via disjoint short  `connecting' paths in $\mathcal{R} - V(\mathcal{W}) - Q$ into an almost spanning cycle. 
     Let $(p_1,p_2), (p_1',p_2')$ be the end pairs of $P_A$. Note that $(p_1,p_2),(p_1',p_2'),(q_1,q_2),(q_1',q_2'),(w_1,w_2),(w_1',w_2')$ are all $\zeta_{**}$-connectable pairs in $H$. 
     
     Using Lemma~\ref{lem:reservoir} three times, connect $(q_1,q_2)$ to $(w_1,w_2)$, then $(w_1',w_2')$ to $(p_1,p_2)$, and then $(p'_1,p'_2)$ to $(q_1',q_2')$ each via a tight path of length $3(\ell +1)$. 
     We will choose these three tight connecting paths to be internally vertex disjoint from each other and the vertices of $Q$ and $\mathcal{W}$. This is possible as the bounds $|V(Q) \cap \mathcal{R}| \leq \vartheta_{**}^2 n $, $|V(\mathcal{W})| \leq 100\ell^2 \delta  n \ll \vartheta_{**}^2 n$ and $2\cdot 3(\ell+1) \ll  \vartheta_{**}^2 n$ mean that the number of vertices to avoid at each step is at most $ 3\vartheta_{**}^2 n \ll \vartheta_{**}\frac{\vartheta^2_{*} n}{4} \le \vartheta_{**}|\mathcal{R}|/2$. 
     For any vertices $x,y,w,z$, each vertex to avoid can be in at most $|\mathcal{R}|^{3\ell}$ tight $(x,y)$---$(z,w)$-paths on $3(\ell+1)$ edges with internal vertices in $|\mathcal{R}|$. Since there are at least $\vartheta_{**} |\mathcal{R}|^{3\ell+1}/2$ tight paths to choose from, we can therefore always find a path that is internally disjoint from the previously used vertices.
     
     As $|\mathcal{R}| \leq \vartheta_{*}^2 n$ and $|V(\hat{H})\setminus (\mathcal{R} \cup V(Q))| \le \vartheta_{*}^2 n$, the number of vertices not used in $\mathcal{W},Q,P_A$ or the final three `connecting' paths is at most $2\vartheta_{*}^2 n$.
     Thus we may apply Lemma~\ref{lem:absorb} to absorb these unused vertices into $P_A$, turning the almost spanning tight cycle into a tight Hamilton cycle $C$. Replacing $\mathcal{W}$ with $\mathcal{W}'$ (that is, replacing $\mathcal{S}_i$ with $\mathcal{S}_i'$ in each switcher) gives an alternative tight Hamilton cycle $C'$. Since $c(\mathcal{W}) \geq c(\mathcal{W}') + 4\delta n$, we have $c(C) \ge c(C') + 4 \delta n$. In particular, either  $c(C') \le - 2\delta n$, in which case $C'$ has at least $(1/2 + \delta)n$ blue edges, or $c(C) \ge 2\delta n$ in which case $C$ has at least $(1/2 + \delta)n$ red edges.
\end{proof}

\section{Proof of the Key Lemma}\label{sec:key_lemma}
In this section, we prove \cref{lem:key}, the Key Lemma.
In fact, it will be more convenient to prove the following alternative statement of Lemma~\ref{lem:key}.
\begin{restatable}[Key Lemma alternative statement]{lem}{keyalt}
\label{lem:key_lemma_alternative}
       Given \cref{setup},  
      suppose that $H$ contains no $\zeta$-connectable $7\ell$-bounded switchers. Then there exists a function $P:V(H) \rightarrow \{-1,+1\}$ and a constant $c \in \{-2,0,+2\}$ such that
      for all but $\gamma n^3$ robust edges $uxy$,
      \[c(uxy) = P(u) + P(x) + P(y) + c.\]
\end{restatable}
One can check that this is equivalent to the original statement of \cref{lem:key}, where the bipartition is given by $X = P^{-1}(+1)$ and $Y= P^{-1}(-1)$ and the three cases $A,B,C$ correspond to $c = -2,0,+2$ respectively.

\subsection{Preliminaries on robust graphs}

The following lemma gives bounds on the sizes of vertex and edge sets of intersections of the robust subgraphs of our link graphs.  
\begin{lem}\label{lem:robust_intersections}
Let $u,v,w\in V(H)$. Given Setup \ref{setup}, we have
\begin{itemize}
    \item  $|V(R_u) \cap V(R_v)|\geq\frac{n}{\sqrt{2}}$,
    \item $e(R_u\cap R_v) \geq 0.22n^2$,  

\item $|V(R_u)\cap V(R_v)\cap V(R_w)|\geq \left(3\frac{2+\sqrt{2}}{4}-2 \right)n\geq 0.56n$,
\item $e(R_u\cap R_v\cap R_w)\geq 0.09n^2.$
\end{itemize}
\end{lem}
\begin{proof}
Recall that by Proposition~\ref{prop:robust} (i) we have $|V(R_u)|\geq \frac{2+\sqrt{2}}{4} n$ for every $u\in V(H)$.
Let $u,v,w\in V(H)$. Then 
\[
|V(R_u)\cap V(R_v)|\geq |V(R_u)| + |V(R_v)| - n\geq 2\frac{2+\sqrt{2}}{4} n-n=\frac{n}{\sqrt{2}},
\]
and 
\[
|V(R_u)\cap V(R_v)\cap V(R_w)|\geq |V(R_u)|+|V(R_v)|+|V(R_w)| - 2n\geq \left(3\frac{2+\sqrt{2}}{4}-2 \right)n.
\]
Note that by Proposition~\ref{prop:robust} (ii)  we have
\[
e(R_v)\geq  \frac{3}{4}\cdot \frac{n^2}{2}-\frac{(n-|V(R_v)|)^2}{2}\geq \frac{6}{8}\cdot \frac{n^2}{2}-\frac{(n-\frac{2+\sqrt{2}}{4}n)^2}{2}= \frac{3+2\sqrt{2}}{8} \frac{n^2}{2}.
\]
for every $v\in V(H)$.
Therefore, 
\[
e(R_u\cap R_v) \geq e(R_u)+e(R_v)-\binom{n}{2}\geq \left(2 \frac{3+2\sqrt{2}}{8}-1\right) \frac{n^2}{2}\geq 0.22n^2 
\]
and
\[
e(R_u\cap R_v\cap R_w) \geq e(R_u)+e(R_v)+e(R_w)-2\binom{n}{2}\geq \left(3 \frac{3+2\sqrt{2}}{8}-2\right) \frac{n^2}{2}\geq 0.09n^2. \qedhere
\]
\end{proof}

\bigskip

We will also need some results on the abundance of connectable pairs, starting with the following fact.

\begin{fact}[{\cite[Fact 4.1]{3hamcycle}}]
\label{fact:connectable}
For any $\rho > 0$, there are at most $\rho n^3$  triples $(x,y,z)\in V(H)^3$ such that 
$xy\in R_z$ and the pair $xy$ fails to be $\rho$-connectable. 
\end{fact}

\begin{define}\label{def:bad}
     For each vertex $u \in V(H)$, let $\hat{R}_u$ be the subgraph of $R_u$ subgraph obtained by keeping only $\zeta$-connectable pairs.  We say $u$ is \emph{extendable} if $e(R_u)-e(\hat{R}_u)< n^2/100$, and \emph{non-extendable} otherwise.
     
\end{define}
\begin{lem}
\label{lem:badvertices}The number of non-extendable vertices is at most $100\zeta n$.
\end{lem}
\begin{proof}
By Fact~\ref{fact:connectable}, there are at most $\zeta n^3$ triples $(x,y,z)\in V^3$ such that $xy\in R_z$ fails to be $\zeta$-connectable. If the number of non-extendable vertices exceeds $100\zeta n$, then
$$
\sum_{z}\big(e(R_z)-e(\hat{R}_z)\big)> \frac{n^2}{100}100 \zeta n=\zeta n^3,
$$
a contradiction.
\end{proof}

\begin{lem} Let $u,v,w\in V(H)$ be extendable. Then 
\begin{itemize}
\item $e(\hat{R}_u \cap \hat{R}_v)\geq 0.2n^2$,
\item $e(\hat{R}_u \cap \hat{R}_v\cap \hat{R}_w)\geq 0.06n^2$.
\end{itemize}
\label{lem:Rhatint}
\end{lem}
\begin{proof}
By Lemma~\ref{lem:robust_intersections}, $e(R_u \cap R_v)\geq 0.22n^2$ and $e(R_u \cap R_v\cap R_w)\geq 0.09n^2$. Further, since $u, v, w$ are extendable, we have $|E(R_a)\setminus E(\hat{R}_a)| < n^2/100$ for each $a \in \{u,v,w\}$.
Combining these bounds, we obtain
\begin{align*}
e(\hat{R}_u \cap \hat{R}_v) &\geq e(R_u \cap R_v)- |E(R_u)\setminus E(\hat{R}_u)| - |E(R_v)\setminus E(\hat{R}_v)| \\
&\geq 0.22n^2-0.01n^2-0.01n^2=0.2n^2,
\end{align*}
and
\begin{align*}
e(\hat{R}_u \cap \hat{R}_v\cap \hat{R}_w) &\geq e(R_u \cap R_v\cap R_w)- \left(\sum_{x \in \{u,v,w\}}|E(R_x)\setminus E(\hat{R}_x)|\right)\\
&\geq 0.09n^2-0.03n^2=0.06n^2. \qedhere
\end{align*}
\end{proof}

\bigskip

\subsection{Edges that agree}

The goal of this section is to identify pairs of edges that are mutually consistent with one of the three bipartitions in the cases of the Key Lemma. We start by investigating pairs of edges $uxy$, $vxy$ where $xy \in \hat{R}_u \cap \hat{R}_v$ and we will show in~\cref{lem:c_uv} that most of these pairs of edges `agree', that is they are consistent with the same bipartition. However, there are possibly a small number of these pairs of edges that we are unable to say anything about directly. We describe these \emph{non-agreeable} edges via the following definition.

\begin{define} For a graph $G$, call a component of $G$ \emph{non-agreeable} if it is possible to label each edge using labels in $\{-2,0,+2\}$ so that not all labels receive the value $0$, but the labels on each path on 3 edges sum to $0$. An edge is \emph{non-agreeable} if it is in a non-agreeable component. 
\end{define}

Note that a component of $G$ that is agreeable must contain a path on 3 edges. We will show that there are only linearly many non-agreeable edges in a graph.

\begin{lem}
\label{annoying}
A non-agreeable component on $n$
vertices has at most $6n$ edges. 
\end{lem}

\begin{proof}
Let $G$ be a non-agreeable component, and fix a labelling
$w:E(G)\to\{-2,0,2\}$ witnessing this. Choose a longest path
$P=v_1v_2\cdots v_k$, and let $a_i := w(v_iv_{i+1})$. If $k \le 3$, then $G$ is a star or triangle and $e(G) \le n$. Therefore we will assume $k \ge 4$. 

If $P$ contains two consecutive edges with weight 0, then (as every copy of $P_3$ has weight 0) every edge of $P$ has weight 0. Now any edge incident to $P$ also has weight 0 (as it forms a 3 edge path with two edges already labelled 0). By induction on the distance of an edge from $P$, the same argument shows
that every edge in $G$ has weight $0$, a contradiction. So $P$ contains an edge with non-zero weight. So, up to permuting the values $2,0,-2$,
the weights along $P$ are $2,0,-2,2,0,-2,\ldots$. 

Let $3\le i\le k-2$. We show that $d(v_i) \le 4$. First suppose that
$x \in X := V(G) \setminus P$ is adjacent to $v_i$. Then both
$v_{i-2}v_{i-1}v_ix$ and $xv_iv_{i+1}v_{i+2}$ are paths of length three.
Hence $w(v_ix)=-(a_{i-2}+a_{i-1})=a_i$ and $w(v_ix)=-(a_i+a_{i+1})=a_{i+2}$, a contradiction as $a_i\ne a_{i+2}$. Now let
$e=v_iv_j$ be a chord with $|i-j|\ge 3$. Here the paths
$v_{i-2}v_{i-1}v_iv_j$ and $v_jv_iv_{i+1}v_{i+2}$ give a contradiction as above. So $N(v_i) \subseteq \{v_{i-2},v_{i-1},v_{i+1},v_{i+2}\}$, as required. By maximality of $P$, we obtain in addition that each vertex of $X$ is necessarily a neighbour of $v_2$ or $v_{k-1}$. 

Note that $v_1$ and $v_k$ cannot have a neighbour outside $P$, as otherwise $P$ is not the longest path. Thus by the observations above, $N(v_1) \subseteq \{v_2,v_3,v_{k-1},v_k\}$ and so $d(v_1) \le 4$ and correspondingly $d(v_k) \le 4$. 

Now consider $W = N(v_2) \setminus V(P)$. We will
show that vertices in $W$ are only adjacent to $v_2$. Suppose not and let $x \in W$ and $y \ne v_2$ be a neighbour of $x$. Note that $y \ne v_1$ as $x \not\in N(v_1) \subseteq V(P)$. Moreover $y \ne v_3$: if $3\le k-2$ then $x \not\in N(v_3) \subseteq V(P)$, and if $3=k-1$ then $y = v_3$ would result in the path $v_1v_2xv_3v_4$, which is longer than $P$. Thus both $v_1v_2xy$ and $v_3v_2xy$ are copies of $P_3$ and have weight 0. In particular, $w(v_1v_2) = w(v_2v_3)$, which is a contradiction as $a_1 \ne a_2$.
By a corresponding argument, vertices in $W' = N(v_{k-1})\setminus V(P)$ are only adjacent to $v_{k-1}$.

We conclude that every edge of $G$ contains a vertex of $P$. Hence, we have a crude bound of
$$e(G) \le \sum_{i=1}^k d(v_i) \le 4(k-2) + 2n  \le 6n.$$

\end{proof}
With a little more care it is possible to obtain a better bound on the number of non-agreeable edges, but the crude bound of $6n$ will suffice for our purposes.

\bigskip

For agreeable edges $xy \in \hat{R}_u \cap \hat{R}_v$, we are able to prove the following important result regarding a relationship between the colours of $uxy$ and $vxy$. This will naturally lead to the definition of the bipartition needed for the Key Lemma.

\begin{lem}\label{lem:c_uv}
Given Setup~\ref{setup}, assume $H$ contains no $\zeta$-connectable $7\ell$-bounded switchers. For each $u,v \in V(H)$ with $u,v$ extendable, the following holds. All agreeable edges $xy\in \hat{R}_u \cap \hat{R}_v$, satisfy the same one of the following: 
\begin{enumerate}[(i)]
\item $uxy$ is red and $vxy$ is blue;
\item $uxy$ and $vxy$ are the same colour; or
\item $uxy$ is blue and $vxy$ is red.
\end{enumerate}
 In particular, there is a constant $c_{uv} \in \{+2,0,-2\}$ such that  $c(uxy) - c(vxy) = c_{uv}$ for all agreeable edges $xy\in\hat{R}_u \cap \hat{R}_v$.
\end{lem}
\begin{proof}
    Fix $u,v \in V(H)$. We will show that there is a constant $c_{uv} \in \{+2,0,-2\}$  such that  $c(uxy) - c(vxy) = c_{uv}$ for all edges $xy\in\hat{R}_u \cap \hat{R}_v$ except for non-agreeable edges. Then $c_{uv} = 2$ implies case (i), $c_{uv} = 0$ implies case (ii) and $c_{uv} = -2$ implies case (iii). To that end, let $f(xy) = c(uxy) - c(vxy)$ for $xy \in \hat{R}_u \cap \hat{R}_v$. Our goal is to show that $f$ is a constant function (on agreeable edges).
    Note that since $c(uxy), c(vxy)$ are equal to $\pm 1$ we have  $f(xy) = c(uxy) - c(vxy) \in \{-2, 0, +2\}$ for all $xy$. 
    \begin{claim}
    \label{claim:disjpath}
    Let $abcd$ and $wxyz$ be vertex disjoint paths in $\hat{R}_u\cap \hat{R}_v$. Then
    $$
             f(ab) + f(bc) + f(cd) = f(wx) + f(xy) + f(yz).
    $$
    \end{claim}
    \begin{claimproof}
    Let $abcd$ and $wxyz$ be vertex disjoint paths in $\hat{R}_u\cap \hat{R}_v$. 
    Consider the potential switcher $(\mathcal{S},\mathcal{S}')$ where $\mathcal{S} = \{abucd, wxvyz\}$ and $\mathcal{S}' = \{abvcd, wxuyz\}$.
    This is a `short' potential switcher, as shown in~\cref{fig:short_switcher} earlier.
     Since $H$ contains no $\zeta$-connectable $7\ell$-bounded switchers, we know that
     $c(\mathcal{S}) = c(\mathcal{S}'),$ which implies the claim.
     \end{claimproof}
         \begin{claim}
    \label{claim:path5}
    Let $abcde$ be a path in $\hat{R}_u\cap \hat{R}_v$. Then $f(ab)  = f(de)$.
    \end{claim}
\begin{claimproof}
Since $u,v\in V(H)$ are extendable, by Lemma~\ref{lem:Rhatint} the graph $\hat{R}_u \cap \hat{R}_v$ contains quadratically many edges, and thus there is some path $wxyz$ in $\hat{R}_u\cap \hat{R}_v$ vertex disjoint from $abcde$. Applying Claim~\ref{claim:disjpath} to the pairs $abcd, wxyz$ and $bcde,wxyz$ yields a pair of equations, which together show that $f(ab) + f(bc) + f(cd) = f(bc) + f(cd)+ f(de)$. Cancelling terms gives $f(ab) = f(de)$. 
\end{claimproof}

    Let $\mathcal{T}$ be the tree with vertices $z_1,z_2,\ldots, z_{10}$ and edges $z_1z_2, z_2z_3, z_3z_4,z_4z_5,z_5z_6$, $z_6z_7$, $z_7z_8$, $z_4z_9$ and $z_5z_{10}$, so $\mathcal{T}$ consists of a path on $7$ edges with two pendant edges attached to the 4th and 5th vertices -- see Figure~\ref{fig:treeT}.

\begin{figure}[h]
    \centering
    \includegraphics[scale = 1]{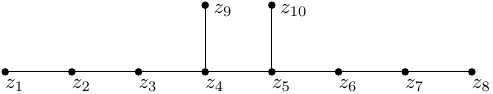}
    \caption{The tree $\mathcal{T}$ on ten vertices}
    \label{fig:treeT}
\end{figure}
            \begin{claim}
    \label{claim:T}
    Let $C$ be a connected component of $\hat{R}_u\cap \hat{R}_v$ containing a copy of $\mathcal{T}$. Then $f(e)=f(e')$ for all $e,e'\in E(C)$.
    \end{claim}
\begin{claimproof}
Let $T$ be a copy of $\mathcal{T}$ in $C$.
 Applying Claim~\ref{claim:path5}, for the edges of $T$ eight times gives
    \begin{align*}f(z_3z_4) &= f(z_6z_7) = f(z_4z_{9}) = f(z_1z_2) = f(z_4z_5)= f(z_7z_8) = f(z_5z_{10}) = f(z_2z_3) = f(z_5z_6).  
    \end{align*}
    In particular, $f$ is constant on $T$. 
    
    Consider an edge $ab$ in $C$.  Let $P$ be a shortest path from $ab$ to $T$; and suppose without loss of generality $P$ starts at $b$. By the structure of $T$, no matter which vertex in $T$ the path $P$ ends on we can extend $P$ to a path $P'$ starting with $ab$ and ending with  an edge $z_iz_j \in T$ such that the number of edges in $P$ is $1 \pmod{3}$. Thus $f(ab) = f(z_iz_j)$, by repeatedly applying Claim~\ref{claim:path5} along $P'$. Since the edge $ab$ was arbitrary, we conclude that $f$ is constant on $C$.
\end{claimproof}
   Since $u,v\in V(H)$ are extendable, Lemma~\ref{lem:Rhatint} implies that the graph $\hat{R}_u \cap \hat{R}_v$ has quadratically many edges. Hence $\hat{R}_u \cap \hat{R}_v$ contains a copy of $\mathcal{T}$ by the classical Erd\H{o}s--Stone--Simonovits' theorem~\cite{esimonovits66, estone46}. Let $C$ be a connected component containing a copy of $\mathcal{T}$ in $\hat{R}_u\cap \hat{R}_v$.
   By Claim~\ref{claim:T}, $f$ is constant on $C$. Let $c_{uv}$ be the value of $f$ on $C$. Now consider any component $C' \neq C$. As $C'$ is agreeable, every edge of $C'$ is contained in a 3-edge path. Let $P=z_1z_2z_3z_4$ be any length 3 path in $C'$.
   Applying Claim~\ref{claim:disjpath} to $P$ and a path of length 3 in $C$, we have 
   $$
   f(z_1z_2) + f(z_2z_3)+f(z_3z_4) = 3c_{uv}.
   $$
    If $c_{uv} = \pm 2$, then $f(e) = c_{uv}$ for all $e \in P$, which completes the proof in this case. If $c_{uv} = 0$, then every path in $C'$ has zero sum. But as $C'$ is agreeable, the only way this is possible is for $f(e) = c_{uv} = 0$ for all $e \in C'$. 
\end{proof}

We are now able to define the bipartition needed in the Key Lemma.
\begin{cor}\label{cor:c(uxy) - c(vxy)}
Given Setup~\ref{setup}, assume $H$ contains no $\zeta$-connectable $7\ell$-bounded switchers. Then there is a function $P:V(H) \rightarrow \{-1,+1\}$ such that for all extendable vertices $u,v$ and all agreeable edges $xy \in \hat{R}_u \cap \hat{R}_v$, we have
\[P(u) - P(v) =c_{uv} =  c(uxy) - c(vxy).\]
\end{cor}
\begin{proof} 
Consider any three distinct extendable vertices $u,v,w$. By Lemma~\ref{lem:Rhatint}, we have that $\hat{R}_u \cap \hat{R}_v \cap \hat{R}_w$ contains quadratically many edges, and so contains a component with quadratically many edges.
Let $xy$ be an edge of such a component. By Lemma~\ref{annoying} the edge $xy$ is agreeable in each of $\hat{R}_u \cap \hat{R}_v$, $\hat{R}_u \cap \hat{R}_w$ and $\hat{R}_v \cap \hat{R}_w$. 
In particular, by Lemma~\ref{lem:c_uv} we have $c(uxy) - c(vxy) = c_{uv}$, $c(uxy) - c(wxy) = c_{uw}$, and  $c(vxy) - c(wxy) = c_{vw}$. Thus $c_{uw} = c_{uv} + c_{vw}$.

For any extendable vertices $u,v$, we write $u\sim v$ if $c_{uv} = 0$. Since $c_{uw} = c_{uv} + c_{vw}$ for all extendable vertices $u,v,w$, this $\sim$ defines an equivalence relation. Moreover, we know that $c_{uv} \in \{-2,0,+2\}$ for all $u,v$ by Lemma~\ref{lem:c_uv}. Thus if $u \not \sim v$ and $v \not \sim w$, then $c_{uw} = c_{uv} + c_{vw} = \pm2 + \pm 2 \in \{-4,0,4\}$. And so, $c_{uw} = 0$, that is $u \sim w$. In particular, this implies that there are no three vertices all in different equivalence classes, and so there are at most two equivalence classes. 

If there is only one equivalence class then we can take  $P: V(H) \rightarrow \{-1,+1\}$ to be constant. Otherwise, there are two equivalence classes $A$ and $B$. Fix $a\in A$ and $b \in B$. Since $a \not\sim b$, we know that $c_{ab} \ne 0$. 
If $c_{ab} = -2$, set $P$ to be $-1$ on $A$ and $+1$ on $B$. Otherwise if $c_{ab} = +2$, set $P$ to be $+1$ on $A$ and $-1$ on $B$. Set $P$ on $V(H)\setminus (A \cup B)$ to arbitrarily say $+1$.

Finally, we check that $P$ has the required properties. Let $s,t \in V(H)$ be extendable vertices. If $s,t$ are in the same class, then $P(s) - P(t) = 0 = c_{st}$. If $s\in A$, $t \in B$, then \[P(s) - P(t) = c_{ab} = c_{as} + c_{sb} = c_{as} + c_{st} + c_{tb} = 0 + c_{st} + 0. \qedhere \]
\end{proof}

\bigskip

We have now obtained the bipartition and shown that certain colour agreement occurs in a local setting, and so our next goal is to prove that the colours agree in a more global sense. 
The following lemma is a first step towards this goal. In short, it shows that if pairs $xy$ and $x'y'$ are linked by a certain kind of gadget (shown in~\cref{fig:xycycle}) then we obtain corresponding colours on $uxy$ and $ux'y'$, for certain $u$. The gadget that we use is related to the long switchers: if the gadget is present then a whole family of potential long switchers can be found. Since we are assuming that none of these are colour-changing switchers, we can use these to deduce properties of the colouring. 

\begin{lem}\label{lem:colours_are_consistent}
      Given \cref{setup},  
      suppose that $H$ contains no $\zeta$-connectable $7\ell$-bounded switchers. Let $x,y,x',y'$ be distinct extendable vertices in $H$. If there exist distinct extendable vertices $w_0,w_1,w',z,z_1',z_2'$, and a set 
    $$U \subseteq V(H) \setminus \left(\{x,y,x',y',w_0,w_1,w',z,z'_1,z'_2 \}\right)$$ of extendable vertices of size $|U|> (\ell +3)/2$ such that the graph $\bigcap_{u \in U}\hat{R}_u$ contains 
    \begin{enumerate}[(i)]
        \item two vertex disjoint paths each on $\ell-1$ extendable vertices,
         \[P \coloneqq p_1p_2\cdots p_{\ell-1} \qquad \text{and} \qquad P' \coloneqq p'_1p'_2\cdots p'_{\ell-1},\]
        where the vertices are distinct from each other and from $x,y,x',y',w_0,w_1,w',z,z'_1,z'_2 $, and
         \item the paths $w_0w_1xyPx'y'z_1'z_2'$ and $zyxP'y'x'w'$ (see Figure~\ref{fig:xycycle}),
    \end{enumerate}
     then for all $u \in U$, we have
    \[c(uxy) - P(x) - P(y) = c(ux'y') - P(x') - P(y'),\] where $P$ is the function given by Corollary~\ref{cor:c(uxy) - c(vxy)}.
\end{lem}
    \begin{figure}[h]
        \centering
        \includegraphics[scale=1]{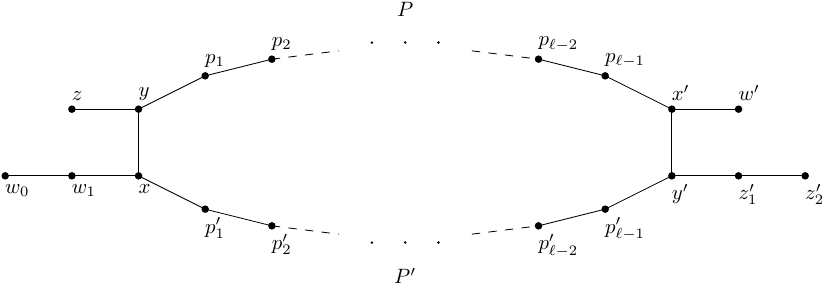}
        \caption{The vertices and edges that must be present in  $\bigcap_{u \in U}\hat{R}_u$ for Lemma~\ref{lem:colours_are_consistent}.}
        \label{fig:xycycle}
    \end{figure}
\begin{proof}
    Let $x,y,x',y'$ be distinct extendable vertices such that there exist vertices $w_0,w_1,w',z,$ $z'_1,z'_2$ and a set $U \subseteq V(H) \setminus \left(\{x,y,x',y',w_0,w_1,w',z,z'_1,z'_2 \}\right)$
        satisfying the conditions of the lemma.
    For each pair of vertices $(v,v') \in \{(x,x'),(y,y')\}\cup\{(p_i,p_i'):1 \le i \le \ell-1\}$ in turn, pick a path $s^{(v)}_1s^{(v)}_2s^{(v)}_3s^{(v)}_4$ on 3 edges in an agreeable component of $\hat{R}_v \cap \hat{R}_{v'}$ disjoint from the vertices $\{w_0, w_1, w', z, z'_1, z'_2\}\cup V(P) \cup V(P')$ and the vertices used in previously constructed paths. Note that these paths exist since 
     $\hat{R}_v \cap \hat{R}_{v'}$ contains quadratically many edges by Lemma~\ref{lem:Rhatint} and the number of non-agreeable edges is at most $6n$ by Lemma~\ref{annoying}.
     
     Now define a pair of tight paths in $H$: $Q_v = s^{(v)}_1s^{(v)}_2vs^{(v)}_3s^{(v)}_4$ and  $Q_{v}'=s^{(v)}_1s^{(v)}_2v's^{(v)}_3s^{(v)}_4$. 
     Note that 
     \begin{equation}
     \label{eq:cQvQv'}
     c(Q_v) - c(Q_{v}') = 3P(v) - 3P(v')
     \end{equation}
     by construction, Corollary~\ref{cor:c(uxy) - c(vxy)} (applicable as the vertices $v,v'$ are extendable).
    
    We will define a family of potential switchers. 
    First we need some notation. Given a positive integer $k$, a tuple $\vec{u} = (u_1,\ldots,u_{k})$ of distinct vertices from $U$, and a path $Q = q_1q_2\ldots q_{2k+2} \in \bigcap_{u \in U} \hat{R}_u$, define $Q_{\vec{u}}$ to be the tight 3-path in $H$ given by
    \[Q_{\vec{u}} = (q_1\cdots q_{2k+2})_{(u_1\ldots u_k)}\coloneqq q_1q_2u_1q_3q_4u_2q_5\cdots   q_{2k}u_{k}q_{2k+1}q_{2k+2}.\]
    See Figure~\ref{fig:Qu} for a diagram.
    \begin{figure}[h]
        \centering
        \includegraphics[scale=1]{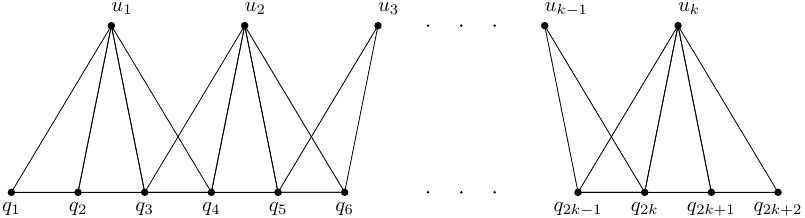}
        \caption{The 3-uniform path $Q_{\vec{u}}$ built from the path $Q$ in $\bigcap_{u \in U} \hat{R}_u$.}
        \label{fig:Qu}
    \end{figure}

    We define two potential switchers for each tuple $\vec{u} = (u_1,\ldots,u_{(\ell+3)/2})$ of $(\ell+3)/2$ distinct vertices in $U$. 
    First consider the following  two sets of paths:
    \[S_{\vec{u}} = (w_0w_1xyPx'w')_{\vec{u}} ~ \cup ~ Q'_{y} ~ \cup \bigcup_{1\le i\le\ell-1} Q_{p_i}',\]
    \[S'_{\vec{u}} = (w_0w_1xP'y'x'w')_{\vec{u}} ~ \cup ~ Q_{y} ~ \cup \bigcup_{1\le i\le\ell-1} Q_{p_i}.\]
    Note that $S_{\vec{u}}$ and $S'_{\vec{u}}$ are both sets of vertex-disjoint tight 3-paths in $H$ by definition, and the paths in $S_{\vec{u}}$ and $S'_{\vec{u}}$ span the same vertex set. Moreover, the initial and final pairs of vertices in the paths are all $\zeta$-connectable and are the same in corresponding paths in $S_{\vec{u}}$ and $S'_{\vec{u}}$. Thus $(S_{\vec{u}},S'_{\vec{u}})$ is a potential switcher. Specifically, it is a potential long switcher, as in~\cref{fig:long_switcher}. Since $H$ contains no $\zeta$-connectable $7\ell$-bounded switchers, we can conclude that the edge colour sums must be equal, that is,
    \begin{equation}
    \label{equation: cSu}
    c( S_{\vec{u}}) = c( S'_{\vec{u}}).
    \end{equation}

    Next consider the following  two sets of paths:    
    \[T_{\vec{u}} = (zyPx'y'z_1'z_2')_{\vec{u}} ~ \cup ~ Q_{x} ~ \cup \bigcup_{1\le i\le\ell-1} Q_{p_i}',\]
    \[T'_{\vec{u}} = (zyxP'y'z'_1z_2')_{\vec{u}} ~ \cup ~ Q'_{x} ~ \cup \bigcup_{1\le i\le\ell-1} Q_{p_i}.\]
    The tight paths in $T_{\vec{u}}$ and $T'_{\vec{u}}$ span the same vertex set and the initial and final pairs of vertices are $\zeta$-connectable and the same in corresponding paths. Thus $(T_{\vec{u}},T_{\vec{u}}')$ is a potential (long) switcher as well. Since $H$ contains no $\zeta$-connectable $7\ell$-bounded switchers, we can conclude that 
        \begin{equation}
    \label{equation: cTu}
    c(T_{\vec{u}}) = c( T'_{\vec{u}}).
    \end{equation}
    Combining \eqref{equation: cSu} with \eqref{equation: cTu} gives
    \[
    c(S_{\vec{u}}) - c(T_{\vec{u}}) - c( S'_{\vec{u}}) + c( T_{\vec{u}}') = 0.
    \]
    This holds for all choices of $\vec{u}$. In particular, it holds if we rotate the entries of $\vec{u}$ on any term. Fix arbitrary distinct vertices $u_1,u_2,\ldots,u_{(\ell+3)/2} \in U$ and let $$j(\vec{u}) := (u_j,u_{j+1},\ldots,u_{(\ell+3)/2},u_1,\ldots,u_{j-1}).$$
    Then 
        \begin{align} 0 &= \sum_{j=1}^{(\ell+3)/2}\left(
    c(S_{j(\vec{u})}) - c(T_{(j+1)(\vec{u})}) - c( S'_{j(\vec{u})}) + c( T'_{j(\vec{u})}) \right)\nonumber\\
    &= \sum_{j=1}^{(\ell+3)/2}\left(
   c(S_{j(\vec{u})}) - c(T_{(j+1)(\vec{u})})\right) + \sum_{j=1}^{(\ell+3)/2}\left( -c( S'_{j(\vec{u})}) +c( T'_{j(\vec{u})}) \right). \label{eq:edge_score_sum}
    \end{align}

    Observe that $ S_{j(\vec{u})}$ and $T_{(j+1)(\vec{u})}$ have many common edges using vertices of $P$. Cancelling these common edges we see that
    \begin{multline*}
    c(S_{j(\vec{u})}) - c(T_{(j+1)(\vec{u})})  \\ = c(w_0w_1u_j) + c(w_1u_jx) + c(u_jxy) + c(xyu_{j+1}) + c(u_{j-1}x'w') + c(Q'_y)
    \\- c(zyu_{j+1}) - c(u_{j-1}x'y') - c(x'y'u_{j}) - c(y'u_{j}z_1') - c(u_jz_1'z_2')  -c(Q_x).
    \end{multline*}
   Similarly,
      \begin{multline*}
    - c( S'_{j(\vec{u})}) +c( T'_{j(\vec{u})})  \\ = -c(w_0w_1u_j) - c(w_1u_jx) - c(y'u_{j-1}x') - c(u_{j-1}x'w') - c(Q_y)
    \\+ c(zyu_{j}) + c(yu_jx) + c(y'u_{j-1}z_1') + c(u_{j-1}z_1'z_2')  + c(Q'_x).
    \end{multline*}
    Substituting these expressions back into \eqref{eq:edge_score_sum} and cancelling (taking into account the sum over $j$), we obtain
      \begin{align} 0 
    &= \sum_{j=1}^{(\ell+3)/2}\left( 3c(u_jxy) - 3c(u_jx'y') + c(Q'_y) - c(Q_y) - c(Q_x) + c(Q'_x)\right) \nonumber \\
    & = \sum_{j=1}^{(\ell+3)/2}\left( 3c(u_jxy) - 3c(u_jx'y') + 3P(y') - 3P(y) - 3P(x) + 3P(x')\right),
    \label{eq:edge_score_simplified}
    \end{align}
   where in the last equality we used \eqref{eq:cQvQv'}. Since \eqref{eq:edge_score_simplified} holds for arbitrary choices of vertices $u_1,\ldots, u_{(\ell+3)/2} \in U$ and $|U| > (\ell+3)/2$, by swapping which choice of elements of $U$ is used we are able to conclude that each term of the summation is equal, which implies that, for all $u \in U$, we must have
    \[ 3c(uxy) - 3c(ux'y') + 3P(y') - 3P(y) - 3P(x) + 3P(x') = 0.\]
    Dividing by three and rearranging gives
    \[ c(uxy) - P(x) - P(y) = c(ux'y') - P(y') - P(x')\]
    as required.
\end{proof}

The following easy corollary of Lemma~\ref{lem:colours_are_consistent} will be more convenient to apply. 
\begin{cor}
\label{cor:colours_are_consistent}
    Given \cref{setup},  
      suppose that $H$ contains no $\zeta$-connectable $7\ell$-bounded switchers. Let $x,y,x',y'$ be distinct extendable vertices in $H$. Suppose that there exists a set $U \subseteq V(H) \setminus \left(\{x,y,x',y' \}\right)$ of extendable vertices of size $|U|> (\ell +3)/2$ such that the graph $\bigcap_{u \in U}\hat{R}_u$ contains $\{xy, x'y'\}$ and two $y$-$x'$ paths $P,Q$ and two $x$-$y'$ paths $P',Q'$ each on $\ell+1$ extendable vertices, such that $P,Q,P',Q'$ are internally vertex disjoint from each other and $x,y,x',y'$ (see  Figure~\ref{fig:xycycle2}). 
   Then for all $u \in U$, we have
    \[c(uxy) - P(x) - P(y) = c(ux'y') - P(x') - P(y'),\] where $P$ is the function given by Corollary~\ref{cor:c(uxy) - c(vxy)}.
\end{cor}
    \begin{figure}[h]
        \centering
        \includegraphics[scale=1]{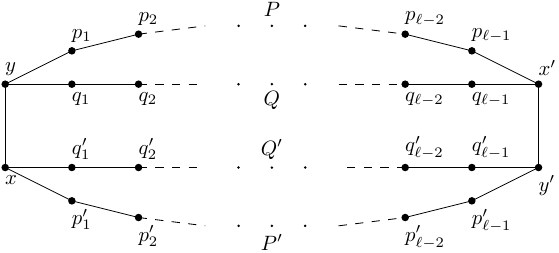}
        \caption{The vertices and edges that must be present in  $\bigcap_{u \in U}\hat{R}_u$ for Corollary~\ref{cor:colours_are_consistent}.}
        \label{fig:xycycle2}
    \end{figure}
\begin{proof}
Let $x,y,x',y'$ be distinct extendable vertices and let $U \subseteq V(H) \setminus \left(\{x,y,x',y' \}\right)$ be a set of extendable vertices of size $|U|> (\ell +3)/2$. Suppose that the graph $\bigcap_{u \in U}\hat{R}_u$ contains the following four 2-paths on $\ell+1$ extendable vertices which are internally vertex disjoint from each other and $x,y,x',y'$:
         \begin{gather*}P \coloneqq yp_1p_2\cdots p_{\ell-1}x' , \qquad P' \coloneqq xp'_1p'_2\cdots p'_{\ell-1}y',  \\
         Q \coloneqq yq_1q_2\cdots q_{\ell-1}x'  \qquad \text{and} \qquad Q' \coloneqq xq'_1q'_2\cdots q'_{\ell-1}y',
         \end{gather*} 
as in Figure~\ref{fig:xycycle2}.

Let
\[
\widetilde P=p_1\cdots p_{\ell-1},
\qquad
\widetilde P'=p'_1\cdots p'_{\ell-1},
\]
and choose
\[
w_0=q'_2,\quad w_1=q'_1,\quad w'=q_{\ell-1},
\quad z=q_1,\quad z'_1=q'_{\ell-1},\quad z'_2=q'_{\ell-2}.
\]
Then $\widetilde P,\widetilde P'$ together with these vertices satisfy the assumptions of Lemma~\ref{lem:colours_are_consistent}. The conclusion therefore follows immediately from the lemma.
\end{proof}

\subsection{From local agreement to global agreement}\label{sec:puttingtogether}

In this section, we will put together all of our results obtained so far to prove the Key Lemma. In particular, we will show that almost all pairs $xy$, $x'y'$ can be linked by chaining together the long gadgets of~\cref{cor:colours_are_consistent} with the local property of~\cref{cor:c(uxy) - c(vxy)}, to get the global colour agreement we desire.

First, we will need another connecting lemma.

\begin{lem}[Another Connecting Lemma]
\label{lem:anotherconnecting}
Given Setup~\ref{setup}, there exists a constant $\mu=\mu(\beta,\ell)>0$ such that 
for every subset of vertices $T\subseteq V(H)$ of size $|T|\leq \mu n$ the following holds.  
Let $x,a\in V(H)$ be distinct vertices and let 
$$
\emptyset \neq W\subseteq \{w\in V: x,a\in V(R_w)\}.
$$
Then there exists $U\subseteq W$ of size $|U|\geq  \mu |W|$ such that there exists an $x$-$a$ path $xr_1\cdots r_{\ell-1}a$ of length $\ell$ in $\cap_{u\in U} R_u$ with $r_i\notin T$ for all $i\in[\ell-1]$.
\end{lem}
\begin{proof}
Set 
\[
\mu:=  \frac{\beta}{3\ell} \left( \frac{2+\sqrt{2}}{4}\right)^{(\ell-1)}.
\] 
For any $(\ell-1)$-tuple of vertices $\vec{r} = (r_1, \dots, r_{\ell-1}) \in V^{\ell-1}$, define
\[
D(\vec{r}) := \{\, w \in W : x r_1 \cdots r_{\ell-1} a \text{ is a path in } R_{w} \,\}.
\]
By Setup~\ref{setup},  for every $w$, $R_w$ is $(\beta,\ell)$-robust and $V(R_w) \ge \frac{2+\sqrt{2}}{4}n$ by~\cref{prop:robust}. Thus we derive that 
\begin{equation}
\label{sumDr}
\sum_{\vec{r} \in V^{\ell-1}} |D(\vec{r})|
\geq |W|\cdot \beta \left(\frac{(2+\sqrt{2})n}{4}\right)^{\ell-1} = |W| \cdot 3 \mu\ell n^{\ell-1}.
\end{equation}
By averaging, $|D(\vec{r})|$ must be large for many tuples $\vec{r}$. More precisely, the number of tuples $\vec{r}\in V^{\ell-1}$ such that $|D(\vec{r})|\geq \mu|W| $ is at least $2\mu \ell n^{\ell-1}$, as otherwise
\begin{align*}
\sum_{\vec{r} \in V^{\ell-1}} |D(\vec{r})|&= \sum_{\substack{\vec{r} \in V^{\ell-1}\\ |D(\vec{r})|\geq \mu |W|}} |D(\vec{r})| + \sum_{\substack{\vec{r} \in V^{\ell-1}\\ |D(\vec{r})|< \mu|W|}} |D(\vec{r})|
\\ 
&< |W|\cdot 2\mu\ell  n^{\ell-1}+ \mu|W| \cdot n^{\ell-1} \\
&\leq |W|\cdot 3\mu\ell n^{\ell-1}, 
\end{align*}
contradicting \eqref{sumDr}.  For any $T\subseteq V(H)$ of size $|T|\leq \mu n$, the number of $\vec{r}\in V^{\ell-1}$ such that $r_i\in T$ for some $i\in[\ell-1]$ is at most $(\ell-1) |T|n^{\ell-2}\leq \mu\ell n^{\ell-1}$. We conclude that there exists $\vec{r}\in V^{\ell-1}$ such that $|D(\vec{r})|\geq \mu|W|$ and $r_i\notin T$ for every $i\in[\ell-1]$. The set $U:= D(\vec{r})$ satisfies the conclusion of the lemma.
\end{proof}

There are certain edges that are difficult to attach to paths. In particular, we cannot apply~\cref{cor:c(uxy) - c(vxy)} to $uxy$ if $u$ is a non-extendable vertex or if $xy$ is a non-agreeable edge in $\hat{R}_u$. Moreover, we also need to be able to find suitable $v$ such that $uxyv$ is a tight path and $xy$ is agreeable in $\hat{R}_v$. These properties inspire the following definitions.

\begin{defs}
    For $u \in V(H)$, we say a pair of vertices $x, y \in V(H)$ is \emph{$u$-problematic} if for more than $\zeta n$ choices of $v \in V(H)$, we have that $xy$ is a non-agreeable edge in $\hat{R}_{u}\cap \hat{R}_{v}$. 

\label{factprob}
    For a constant $\varepsilon > 0$, an edge $xyz$ in $H$ is said to be \emph{$\eps$-good} if there exists a permutation of $x,y,z$ such that $x,y,z$ are extendable (see~\cref{def:bad}), $xy\in E(R_{z})$ and $xy$ is both $\eps$-connectable 
    and not $z$-problematic.
\end{defs} 

We will show that there are few $u$-problematic pairs for all $u$, and use this to prove that most edges are $\eps$-good for a suitable choice of $\eps$.

\begin{fact}\label{lem:problematic}
Let $u\in V(H)$. The number of $u$-problematic pairs is at most $7n/\zeta$.
\end{fact}
\begin{proof}
The number of triples $(x,y,v)$ such that $xy$ is non-agreeable in $\hat{R}_{u}\cap \hat{R}_{v}$ is at most $n \cdot 6n=6n^2$ by Lemma~\ref{annoying}. If there are more than $7n/\zeta$ many $u$-problematic pairs $x,y$ then there would be more than $\zeta n \cdot 7n/\zeta = 7n^2$ such triples $(x,y,v)$, a contradiction.
\end{proof}

\begin{lem}
\label{lem:goodedges}
Given~\cref{setup}, let $0 < \eps \ll \gamma$. The number of robust edges in $H$ which are not $\eps$-good is at most $\gamma n^3$. 
\end{lem}
\begin{proof}
 By~\cref{fact:connectable}, the number of triples $(x,y,z)\in V(H)^3$ such that $xy\in R_z$ and $xy$ fails to be $\eps$-connectable is at most $\eps n^3$. By~\cref{lem:problematic}, for every $z\in V(H)$ the number of $z$-problematic pairs is at most $7n/\zeta$. 
 By~\cref{lem:badvertices}, the number of non-extendable vertices is at most $100 \zeta n$. 
 Hence the number of robust edges in $H$ which are not $\varepsilon$-good is at most 
\begin{equation*}
\eps n^3+ n \cdot \frac{7n}{\zeta} + 100 \zeta n\cdot n^2 < \gamma  n^3,
\end{equation*}
where we have used that $1/n_0 \ll \zeta \ll \gamma$ from~\eqref{eq:hierarchy}.
\end{proof}

Now we know which undesirable edges from $H$ to discard, we can proceed with proving the key lemma. We restate the lemma below (in its alternative form) for convenience.
\keyalt*

\begin{proof}[Proof]
    Let $P:V(H) \rightarrow \{-1,+1\}$ be as defined in Corollary~\ref{cor:c(uxy) - c(vxy)}.
    Define $g(uxy) := c(uxy) - P(u) - P(x) - P(y)$.   
   By Corollary~\ref{cor:c(uxy) - c(vxy)}, if $xy$ is an agreeable edge in $\hat{R}_u \cap \hat{R}_v$ for some extendable $u,v\in V(H)$ then \begin{equation}
   \label{nonannoying}
   g(uxy) - g(vxy) = c(uxy) - P(u) - c(vxy) + P(v) = 0.
   \end{equation}
  
Let $E_R\subseteq E(H)$ be the set of robust edges in $H$. Further, set $\eps = 1000\mu^{-8}\zeta$, where $\mu$ is taken from~\cref{lem:anotherconnecting}, and note that by \cref{setup} \eqref{eq:hierarchy}, we may ensure $\eps \ll \gamma$. Let $E\subseteq E_R$ be the set of edges which are $\eps$-good. Lemma~\ref{lem:goodedges} implies that
$ |E_R\setminus E|\leq \gamma n^3.$
 
To establish Lemma~\ref{lem:key_lemma_alternative}, it will therefore suffice to show that $g$ is constant on  $E$. In fact, it suffices to show that $g(e)=g(f)$ for vertex-disjoint $e,f \in E$. If $e,f$ intersect then take a third edge $e' \in E$ disjoint from both, which is possible as $|E| = \Theta(n^3)$. If $g(e) = g(e')$ and $g(e') = g(f)$ then we also have $g(e) = g(f)$.

Therefore let $e=vxy$ and $f=v'x'y'$ be disjoint edges in $E$. The `$\varepsilon$-goodness' of $e$ and $f$ are witnessed by the following: $x,x',y,y',v,v'$ are extendable, $xy \in E(R_v)$; $x'y' \in E(R_{v'})$; $xy$ is $\eps$-connectable; $x'y'$ is $\eps$-connectable; $xy$ is not $v$-problematic and $x'y'$ is not $v'$-problematic. Set
$$
W:=\{w\in V: xy \text{ is agreeable in } \hat{R}_w\cap \hat{R}_v\}
$$
and
$$
W':=\{w\in V: x'y' \text{ is agreeable in } \hat{R}_w\cap \hat{R}_{v'}\}.
$$
Since $xy$ is $\eps$-connectable, $x'y'$ is $\eps$-connectable, $e$ is not $v$-problematic and $f$ is not $v'$-problematic, we have that  $|W|\geq (\eps - \zeta)n \ge 999\mu^{-8}\zeta n$ and $|W'|\geq 999\mu^{-8}\zeta n$. Set ${s:=\lfloor 300\mu^{-8}\zeta n \rfloor}$.
By Lemma~\ref{lem:badvertices}, the number of non-extendable vertices is at most $100\zeta n$  and thus we can choose
$w_1, \ldots , w_s \in W$ and $w_1', \ldots , w_s' \in W'$ to be distinct, extendable vertices. Define
\[
I_{a,b} := \left\{ i \in [s] : ab \text{ is agreeable in } \hat{R}_{w_i} \cap \hat{R}_{w_i'} \right\}
\]
for any ordered pair $(a,b)$ of vertices from $V$. By double counting and Lemmas~\ref{lem:robust_intersections} and \ref{annoying}, and since $W$ and $W'$ only contain extendable vertices,
\begin{align}
\label{Iab1}
\nonumber
\sum_{(a,b)\in V^2} |I_{a,b}|
 &\geq \sum_{i=1}^s \left(\left| E(\hat{R}_{w_i}) \cap E(\hat{R}_{w_i'}) \right|-6n \right)
 \geq \sum_{i=1}^s \left(\left| E(R_{w_i}) \cap E(R_{w_i'}) \right|-2\frac{n^2}{100}-6n \right)\\
&\ge  \frac{n^2 s}{100}.
\end{align}
Further, by Lemma~\ref{lem:badvertices}, the trivial bound $|I_{a,b}|\leq s$, and $\zeta \ll \alpha \leq 1$ by \eqref{eq:hierarchy}, it holds that
\begin{align}
\label{Iab2}
\sum_{\substack{(a,b)\in V^2 \\ a \text{ or } b \text{ is non-extendable}}} |I_{a,b}|\leq 100 \zeta n^2 s \leq \frac{n^2s}{10000}.
\end{align}
Thus, by combining \eqref{Iab1} with \eqref{Iab2}, we obtain
\begin{align*}
\sum_{\substack{(a,b)\in V^2 \\ a \text{ and } b \text{ are extendable}}} |I_{a,b}|\geq \sum_{(a,b)\in V^2} |I_{a,b}| \ \ - \sum_{\substack{(a,b)\in V^2 \\ a \text{ or } b \text{ is non-extendable}}} |I_{a,b}| \geq \frac{n^2 s}{100} - \frac{n^2 s}{10000} \geq \frac{n^2 s}{120}.
\end{align*}

Let $(a,b)\in V^2$ such that $a$ and $b$ are extendable, and $|I_{a,b}|\geq s/120\geq 2\mu^{-8}\zeta n$.
We now see that by iteratively applying Lemma~\ref{lem:anotherconnecting} eight times there exists an index set
$I\subseteq I_{a,b}$ of size $|I|\geq \mu^{8} |I_{a,b}|\geq \mu^{8} \cdot 2 \mu^{-8}\zeta n=2\zeta n $ such that
\begin{itemize}
\item $\cap_{i\in I }R_{w_i}$ contains two $y$-$a$ paths $P_1,P_2$ of length $\ell$, 
\item $\cap_{i\in I }R_{w_i}$ contains two $x$-$b$ paths $P_3,P_4$ of length $\ell$, 
\item $\cap_{i\in I }R_{w_i'}$ contains two $y'$-$a$ paths $P_5,P_6$ of length $\ell$, and
\item $\cap_{i\in I }R_{w_i'}$ contains two $x'$-$b$ paths $P_7,P_8$ of length $\ell$. 
\end{itemize}
Further, all of these paths can be chosen so that they are internally
vertex-disjoint and all their internal vertices are extendable. Indeed, let
$B_{\rm ne}$ be the set of non-extendable vertices $v\in V(H)$. By Lemma~\ref{lem:badvertices}, we have
$|B_{\rm ne}|\leq 100\zeta n$.
Since $\zeta\ll \mu$, we may assume that $100\zeta n+10\ell<\mu n$. We construct the paths consecutively. Suppose that, for some
$0\leq i\leq 7$, we have already found paths
$P_1,\ldots,P_i$ with pairwise internally disjoint vertex sets, and
next we want to find the $f$-$g$ path $P_{i+1}$, where
$(f,g)\in \{(y,a),(x,b),(y',a),(x',b)\}$.
Set
\[
T_i :=
B_{\rm ne}
\cup \{x,y,x',y',a,b\}
\cup
\left(\bigcup_{j=1}^i V(P_j)\right).
\]
Then $|T_i|\leq 100\zeta n+6+8\ell<\mu n$. Applying Lemma~\ref{lem:anotherconnecting} with $T=T_i$ gives the
desired path $P_{i+1}$ whose internal vertices avoid $T_i$. Hence
$P_{i+1}$ is internally disjoint from the previously chosen paths and
all of its internal vertices are extendable.

Note that every pair of vertices appearing in the union of these 8 paths is $(|I|/n)$-connectable and thus in particular $\zeta $-connectable since $|I|\geq 2\zeta n$. Therefore  $P_1,P_2,P_3,P_4\in \cap_{i\in I }\hat{R}_{w_i}$, and $P_5,P_6,P_7,P_8\in \cap_{i\in I }\hat{R}_{w_i'}$.

Now, fix some $i\in I$ and let $u=w_i$ and $u'=w_i'$. Since $ab$ is agreeable in $\hat{R}_u\cap \hat{R}_{u'}$ and $u$ and $u'$ are extendable, by \eqref{nonannoying}, 
\begin{equation}
\label{g1}
g(uab)=g(u'ab).
\end{equation} Recall that $x, y, x', y'$ are all extendable and we chose $(a,b) \in V^2$ such that $a$ and $b$ are extendable. We can then apply Corollary~\ref{cor:colours_are_consistent} for the vertices $x,y,a,b$ and the paths $P_1,P_2,P_3,P_4$ to obtain 
\begin{equation}
\label{xyab}
c(uxy)-P(x)-P(y)=c(uab)-P(a)-P(b). 
\end{equation}
Similarly, by applying Corollary~\ref{cor:colours_are_consistent} for the vertices $x',y',a,b$ and the paths $P_5,P_6,P_7,P_8$ we obtain 
\begin{equation}
\label{x'y'ab}
c(u'x'y')-P(x')-P(y')=c(u'ab)-P(a)-P(b).
\end{equation}
Using the definition of $g$, \eqref{xyab} and \eqref{x'y'ab} together imply 
\begin{equation}
\label{g2}
g(uxy)=g(uab) \quad \text{and} \quad g(u'x'y')=g(u'ab).
\end{equation}
Further, since $xy$ is agreeable in $\hat{R}_u\cap \hat{R}_v$ as $u \in W$, and $x'y'$ is agreeable in $\hat{R}_{u'}\cap \hat{R}_{v'}$ as $u' \in W'$, by \eqref{nonannoying}, 
\begin{equation}
\label{g3}
g(vxy)=g(uxy) \quad \text{and} \quad g(u'x'y')=g(v'x'y').
\end{equation}
Combining \eqref{g1} and \eqref{g2} with \eqref{g3}, we conclude    \[g(vxy) = g(uxy) = g(uab) = g(u'ab) = g(u'x'y') = g(v'x'y'),\]
completing the proof of Lemma~\ref{lem:key_lemma_alternative}.    
\end{proof}

\section{Proof of Theorem~\ref{thm:main}}
\label{sec:fews}

We are now ready to complete the proof of the main theorem, restated below for convenience. 
\main*

\begin{proof}[Proof]
    Let $H$ be a 2-coloured 3-graph on $n \geq n_0$ vertices with $\delta_1(H) \geq (\frac{3}{4} + \alpha)\binom{n}{2}$. Our goal is to show that $H$ contains a tight Hamilton cycle with at least $(\frac{1}{2} + \delta)n$ edges of the same colour. We will use the robust graphs and constant hierarchy given by \cref{prop:robust} and \cref{setup}.

    First, apply Lemma~\ref{lem:lots_of_switchers} to obtain $\delta$ such that whenever $H$ contains $4\delta n$ disjoint $7\ell$-bounded $\zeta$-connectable switchers, then $H$ contains a tight Hamilton cycle with at least $(\frac{1}{2} + \delta)n$ edges of the same colour. So we may assume that $H$ contains fewer than $4\delta n$ disjoint $7\ell$-bounded $\zeta$-connectable switchers.

   Let $U$ be the vertex set of a largest set of disjoint $7\ell$-bounded $\zeta$-connectable switchers in $H$. Then, as $\ell\delta \ll \alpha$, we have $|U|\leq 28\ell\delta n \leq \frac{\alpha n}{100}$. Set $V'=V(H)\setminus U$ and $H'=H[V']$ and note that $\delta_1(H')\geq (3/4 + \frac{\alpha}{2})\binom{n'}{2}$. 

We wish to define  $\zeta$-connectable in $H'$ in a way that is harmonious with the definition of $\zeta$-connectable in $H$. To that end, we will show that $R_v[V']$  is a robust graph for all $v \in V'$. Note that for all $x,y \in R_v$ the number of $x,y$ paths of length $\ell$ in $R_v$ containing a vertex in $U$ is at most $|U||V(R_v)|^{\ell-2}$. Thus, since $R_v$ is $(\beta,\ell)$-robust, for all $x,y \in R_v[V']$, the number of $x,y$ paths of length $\ell$ in $R_v[V']$ is at least
\[\beta |V(R_v)|^{\ell-1} - |U||V(R_v)|^{\ell-2} \ge \left(\beta  - \frac{28\delta\ell n}{|V(R_v)|}\right)|V(R_v)|^{\ell-1} \ge \frac{\beta}{2}|V(R'_v)|^{\ell -1},\]
using $\ell\delta \ll \beta$ and $|V(R_v)| \ge \frac{2+\sqrt{2}}{4}n$.
In particular, $R'_v$ is $(\beta/2, \ell)$-robust. Therefore, using $R'_v \coloneqq R_v[V']$ as the robust graphs, we can define $\rho$-connectable in $H'$ in the same way as in~\cref{def:connectable}. According to this definition, we therefore have that  any pair $xy$ that is $2\zeta$-connectable in $H'$ was also $\zeta$-connectable in $H$ (as $|V'| \ge |V|/2$). In particular, any $7\ell$-bounded $2\zeta$-connectable switchers in $H'$ are also $7\ell$-bounded $\zeta$-connectable switchers in $H$. Since every such switcher contained a vertex in $U$, we can conclude that $H'$ contains no $7\ell$-bounded $2\zeta$-connectable switchers.

 
   Replacing $n$, $\alpha$, $\beta$, and $\zeta$ by  $n' \coloneqq|H'|$, $\alpha' \coloneqq \alpha/2$, $\beta' \coloneqq \beta/2$, and $\zeta' \coloneqq 2\zeta$ respectively in Setup~\ref{setup}, we can therefore apply~\cref{lem:key} to $H'$. Let $X\cup Y$ be the partition of $V'$ thus obtained.

   For each of the three cases in Lemma~\ref{lem:key}, let $E_{\textrm{bad}}$ be the set of robust edges to which the case does not apply, so $|E_{\textrm{bad}}| \le \gamma n^3$, and let $W\subseteq V'$ be the set of vertices that are incident to at least $\sqrt{\gamma}n^2$ edges of $E_{\textrm{bad}}$. Note that $|W|\leq \frac{3|E_{\textrm{bad}}|}{\sqrt{\gamma}n^2} \le 3\sqrt{\gamma}n$. Let $E^* \subseteq E_{\textrm{bad}}$ be the set of edges in $E_{\textrm{bad}}$ that do not contain a vertex in $W$.  

    Let $R\subseteq H$ be the sub $3$-graph consisting only of robust edges.
    Then, by Proposition~\ref{prop:robust}, it holds that
    \begin{align*}
    \delta_1(R)&\geq \min_{v\in V}e(R_v)\geq \left(\frac{3}{4}+\frac{\alpha}{2}\right)\frac{n^2}{2}-\frac{(n-|V(R_v)|)^2}{2} \\
    &\geq \frac{3n^2}{8}-\frac{(0.15n)^2}{2}\geq 0.36n^2.
    \end{align*}
  
    Now, let $G\subseteq R$ be the sub $3$-graph on $V$ with edge set $E(G)=E(R)\setminus E^*$. 
     By definition, each vertex is adjacent to at most $\sqrt{\gamma}n^2$ edges in $E^*$, and so 
        \begin{equation*}
 \delta_1(G)\geq \delta_1(R)- \sqrt{\gamma}n^2 \geq 0.36n^2-\sqrt{\gamma}n^2\geq 0.35n^2.
    \end{equation*} 
    Thus, by Theorem~\ref{thm:hamcyclenocolour}, $G$ contains a Hamilton cycle $C$, which is also a Hamilton cycle in $H$. Let  $E_1\cup E_2$ be a partition of $E(C)$, where $E_1$ is the set of edges incident to at least one vertex from $W\cup U$. Note that 
\begin{equation}\label{eq:e1}
    |E_1|\leq 3(|W|+|U|)\leq 9\sqrt{\gamma}n+84\ell\delta  n\leq \frac{n}{100}.
\end{equation}
Moreover, every edge in $E_2$ is a robust edge in $H'$ (as it is robust in $H$ and does not contain a vertex in $U$) and not in $E_{\textrm{bad}}$ (as it neither contains a vertex from $W$ nor is in $E^*$).
   We now separately analyse the three different cases given by the application of Lemma~\ref{lem:key} in $H'$. Given a vertex $v\in V'$, write $X_v:=V(R_v)\cap X$ and $Y_v:=V(R_v)\cap Y$. \\
   
   \noindent \textbf{Cases A and C:} We will give the details for Case A, and note that Case C will follow analogously by swapping the roles of $X$ and $Y$, and red with blue. It suffices to show that $|Y|\le 0.15n$. Indeed, if $|Y|\le 0.15n$ then the number of blue edges in $E_2$ is at most $3|Y| \le 0.45n$, since every blue edge in $E_2$ is incident to a vertex from $Y$ and every vertex from $Y$ is in at most three edges from $C$. Using \eqref{eq:e1}, the total number of blue edges in $C$ is at most $|E_1|+3|Y|\leq0.46n$. Thus, $C$ is (significantly) colour-biased.
      
    Suppose for a contradiction that $|Y|> 0.15n$. Recall that in this case any robust edge entirely in $Y$ must be in $E_{\textrm{bad}}$. Since $|E_{\textrm{bad}}|\leq \gamma n^3$ and $|Y|>0.15n$, there exists $y\in Y$ such that $y$ is incident to at most $100\gamma n^2$ robust edges from $E_{\textrm{bad}}$: if not, there would be at least $100\gamma n^2|Y|/3 \ge 5\gamma n^3$ edges in $E_{\textrm{bad}}$.  Using Proposition~\ref{prop:robust},
    \begin{equation}
    \label{ineq:lowerRya}
    e(R_y[V'])\geq e(R_y)- n|U|\geq\left(\frac{3}{4}+\frac{\alpha}{2}\right)\frac{n^2}{2}-\frac{(n-|X_y|-|Y_y|)^2}{2}-n|U|.
    \end{equation}
    On the other hand, as all robust edges not in $E_{\textrm{bad}}$ intersect $Y$ on 0 or 1 vertices,
        \begin{equation*}
    e(R_y[V'])\leq \frac{|X_y|^2}{2}+100\gamma n^2.
    \end{equation*}
    Combining the previous two equations and recalling $|U| \le \frac{\alpha}{100}n$ gives
            \begin{equation*}
 \left(\frac{3}{4}+\frac{\alpha}{2}\right)\frac{n^2}{2}-\frac{(n-|X_y|-|Y_y|)^2}{2}-\frac{\alpha}{100}n^2 \leq \frac{|X_y|^2}{2}+100\gamma n^2.
    \end{equation*} 
Since $\gamma\ll \alpha$, and $|X_y|+|Y_y| = |V(R_y[V'])|\geq 0.85n$ by Proposition~\ref{prop:robust}, it holds that
            \begin{equation*}
 \left(\frac{3}{4}+\frac{\alpha}{4}\right)\frac{n^2}{2}-\frac{(0.15n)^2}{2}\leq \frac{|X_y|^2}{2}.
    \end{equation*} 
    This implies $|X|\geq |X_y|\geq 0.85n $ and thus $|Y|\leq 0.15n$, a contradiction. \\
    
    \noindent \textbf{Case B:} 
    We assume that $|V'|/2\leq |Y|$ and thus $|Y|\geq 0.49n$; the other case follows analogously. The number of robust edges in $H'[X]$ and $H'[Y]$ in total is at most $|E_{\textrm{bad}}| \le \gamma n^3$. Thus, there exists $y\in Y$ such that $y$ is incident to at most $7\gamma n^2$ robust edges from $H'[Y]$. As above, we have \eqref{ineq:lowerRya}.
    On the other hand,
        \begin{equation}
        \label{ineq:upperRy}
    e(R_y[V'])\leq |X_y||Y_y|+\frac{|X_y|^2}{2}+7\gamma n^2.
    \end{equation}
    Setting $a:=|X_y|/n$, $b:=|Y_y|/n$ and combining \eqref{ineq:lowerRya} with \eqref{ineq:upperRy} and that $|U| \leq \frac{\alpha}{100}n$ we obtain
    \begin{equation*}
    \left(\frac{3}{4}+\frac{\alpha}{2} \right)\frac{1}{2}-\frac{(1 -a-b)^2}{2}-\frac{\alpha}{100} \leq ab+\frac{a^2}{2}+7\gamma .
    \end{equation*}
    Again using that $\gamma \ll \alpha$, we see that 
    \begin{equation}
    \label{ineq:lowerupperRy2}
  \frac{3}{8}+\frac{\alpha}{8}\leq ab+\frac{a^2}{2}+\frac{(1 -a-b)^2}{2}.
    \end{equation}
    Finally, let us prove that inequality \eqref{ineq:lowerupperRy2} cannot hold. Setting $c:=a+b$, we have $c = (|X_y|+|Y_y|)/n$ and so $1 \ge c \geq 0.85$ by Proposition~\ref{prop:robust}. Using this and $a\leq |X|/n\leq 1/2$, we obtain
        \begin{equation*}
    ab+\frac{a^2}{2}+\frac{(1 -a-b)^2}{2}\leq \frac{c^2}{4}+\frac{1}{8}+\frac{(1-c)^2}{2} = \frac{3}{8} - \frac{(3c-1)(1-c)}{4}\leq \frac{3}{8}.
    \end{equation*}
    This contradicts inequality \eqref{ineq:lowerupperRy2}.\end{proof}

\section{Concluding remarks}\label{sec:concl}

In this paper, we showed that the best possible minimum vertex degree conditions forcing colour-biased perfect matchings and colour-biased tight Hamilton cycles in two-coloured 3-graphs align asymptotically. It is natural to consider whether such a phenomenon persists to higher uniformities. As it turns out, for all $k\geq 4$ the (asymptotic) best possible minimum vertex degree condition forcing a colour-biased tight Hamilton cycle in two-coloured $k$-graphs is significantly higher than that of the (asymptotic) best possible minimum vertex degree condition forcing a colour-biased perfect matching.

To see this, first consider the following construction. 
Let $k \geq 3$ and $V = V_1 \sqcup V_2$, where $|V| = n$ is divisible by $2k$ and $|V_1| = (1 - \frac{1}{2k})n$ and $|V_2| = \frac{n}{2k}$.
Let $H$ be the $k$-graph on $V$ whose edge set consists of all possible edges except those with precisely $k-2$ vertices in $V_1$. 
Observe that any tight Hamilton cycle in $H$ must use only edges with at least $k-1$ vertices in $V_1$.
Indeed, there are no tight paths in $H$ containing both edges with at least $k-1$ vertices in $V_1$ and edges with at most $k-3$ vertices in $V_1$. 
Moreover, the large difference between $|V_1|$ and $|V_2|$ clearly ensures the vertices of $V_1$ cannot be covered by any tight path containing an edge from $V_2$. 
Now colour the edges with precisely $k$ vertices in $V_1$ with red, the edges with precisely $k-1$ vertices in $V_1$ with blue, and arbitrarily colour all other edges with red or blue. 
Since $k$ divides $n$, every tight Hamilton cycle $C$ in $H$ is made up of $k$ perfect matchings.
Together with the earlier observation that any tight Hamilton cycle in $H$ must use only edges with at least $k-1$ vertices in $V_1$, each of the perfect matchings making up $C$ must contain $n/2k$ red edges and $n/2k$ blue edges.
Thus, $C$ is colour balanced, i.e., it contains the same number of red and blue edges.
One can calculate that $$\delta_1(H) \geq \left(1 - \frac{(k-1)\left(1 - \frac{1}{2k}\right)^{k-2}}{2k} - o(1)\right)\binom{n}{k-1}.$$ 
Note that the above example was given only for $k=4$ in \cite{color-bias25}. Set $d_k := 1 - \frac{(k-1)\left(1 - \frac{1}{2k}\right)^{k-2}}{2k}$. 
Now, $d_3 \approx 0.72 < 3/4$, which accords with how the example witnessing that Theorem~\ref{thm:main} is (asymptotically) best possible is different to the above example.
However, one can calculate that $d_4 \geq 0.7$ and $d_k \geq 2/3$ for all $k \geq 5$, where the latter follows from $(1 - \frac{1}{2k})^{k-2} \leq(1 + \frac{k-2}{2k})^{-1}$.
Inspection of the values in \cite[Table 1]{color-bias25} together with \cite[Theorem~1.3]{color-bias25} reveal that, for all $k\geq 4$, 
the above example demonstrates the difference between the best possible minimum degree conditions forcing a colour-biased Hamilton cycle and colour-biased perfect matching.

Let $h_1(k,n)$ be the smallest integer $h$ such that every $n$-vertex $k$-graph $H$ with $\delta_1(H) \geq h$ contains a tight Hamilton cycle. Set $$z_k := 1 - \binom{k-1}{\lfloor (k-1)/2\rfloor}\frac{\lceil (k-1)/2 \rceil^{\lceil (k-1)/2 \rceil} (\lfloor (k-1)/2 \rfloor + 1)^{\lfloor (k-1)/2 \rfloor}}{k^{k-1}}.$$  Han and Zhao~\cite{hz16} showed the following.

\begin{thm}[Han and Zhao {\cite[Theorem~1.4]{hz16}}]\label{thm:hz}
    For all $k \geq 3$, we have that $h_1(k,n) \geq \left(z_k + o(1) \right)\binom{n}{k-1}$.
\end{thm} 

Note that $z_k > d_k$ for $k \geq 7$, but that $z_k < d_k$ for $k = 4,5,6$. In light of our earlier example and Theorem~\ref{thm:hz}, we make the following conjecture, which was originally stated only for $k=4$ in \cite{color-bias25}.

\begin{conj}\label{conj:new}
    For all $\alpha > 0$ and $k \geq 4$ there exist $\delta, n_0 > 0$ such that the following holds. Let $H$ be a red/blue coloured $k$-graph on $n \geq n_0$ vertices with $\delta_1(H) \geq (\max\{d_k, z_k\} + \alpha)\binom{n}{k-1}$. Then $H$ contains a tight Hamilton cycle with at least $(\frac{1}{2} + \delta)n$ edges of the same colour.  
\end{conj}

Let $H$ be a $k$-graph on $n$ vertices with $n$ divisible by $k-1$. 
Define a \emph{loose Hamilton cycle in $H$} to be a cyclic ordering of the vertices $v_1, \ldots, v_n$ of $H$ such that $v_i\cdots v_{i+k-1}$ is an edge for all $i \equiv 1 \mod (k-1)$.
Bu\ss, H\`{a}n and Schacht~\cite{bhs2013} proved that for a 3-graph $H$ the asymptotically best possible minimum vertex degree condition forcing a loose Hamilton cycle in $H$ is $\delta_1(H) \geq (\frac{7}{16} + o(1))\binom{n}{2}$.
To see that this result is tight, consider the following example based on \cite[Fact~4]{bhs2013}. For $n \in \mathbb{N}$ divisible by 8,
let $V := V_1 \sqcup V_2 \sqcup V_3$ be a partition of a set of $n$ vertices where $|V_1| = |V_2| = n/8$ and $|V_3| = 3n/4$.
Define $H$ to be the 3-graph whose edge set consists of all edges that intersect $V_1 \cup V_2$ in at least one vertex.
Colour all edges containing one vertex of $V_1$ and two vertices of $V_3$ using red and all edges containing one vertex of $V_2$ and two vertices of $V_3$ using blue. 
Colour all other edges arbitrarily.
Observe that for any loose Hamilton cycle $C$ in $H$, every vertex in $V_1 \cup V_2$ is in at most 2 edges, 
and the number of edges in a loose Hamilton cycle is $n/2$. Hence
every vertex in $V_1 \cup V_2$ must be contained in two edges of $C$ as $|V_1 \cup V_2| = n/4$.
Every vertex in $V_1$ thus contributes $2$ red edges to $C$ and every vertex in $V_2$ contributes $2$ blue edges to $C$ and therefore $C$ is colour balanced. 
One can calculate that $\delta_1(H) = \frac{7}{16}\binom{n}{2} - \frac{n}{32}$. 
We believe the minimum vertex degree condition in the result of Bu\ss, H\`{a}n and Schacht~\cite{bhs2013} should also force a colour-biased loose Hamilton cycle.

\begin{conj}\label{conj:final}
    For all $\alpha > 0$ there exist $\delta, n_0 > 0$ such that the following holds. 
    Let $H$ be a red/blue coloured $3$-graph on $n \geq n_0$ vertices with $\delta_1(H) \geq (\frac{7}{16} + \alpha)\binom{n}{2}$. 
    Then $H$ contains a loose Hamilton cycle with at least $(\frac{1}{4} + \delta)n$ edges of the same colour.
\end{conj}

\section{Acknowledgements}
In the first version of this paper, we had a version of Conjecture~\ref{conj:new} that suggested that the degree threshold was $\delta_1(H)\geq (d_k+\alpha)\binom{n}{k-1}$. Nicolás Sanhueza-Matamala, Richard Lang,  Matías  Pavez-Signé and Camila Zárate-Guerén drew our attention to the paper of 
Han and Zhao~\cite{hz16} showing that this original version of Conjecture~\ref{conj:new} was false. In the meantime, Deng, Tang and Yang~\cite{DengTangYang2026} released a paper also disproving that original conjecture in the case $k=17$ with a different construction.\\

While working on this project, Natalie Behague was initially supported by the European Research Council (ERC) under the European Union Horizon 2020 research and innovation programme (grant agreement No. 947978). 
Felix Clemen's research is supported by a PIMS Postdoctoral Fellowship (PIMS-20260730-PDF).
Natasha Morrison's research is supported by NSERC Discovery Grant RGPIN-2021-02511. Joseph Hyde was supported by UK Research and Innovation Future Leaders Fellowship MR/W007320/2 and is supported by a Leverhulme Early Career Fellowship.

\bibliographystyle{plain} 
\bibliography{refs}

\appendix{}

\section{Proof of Proposition~\ref{prop:robust}}\label{app:robust}

 Recall the statement of Proposition~\ref{prop:robust}.
\robust*

In fact, we can prove the following stronger proposition, which requires a weaker minimum degree condition on $H$. We include this more general version as it may be useful for future applications. 
\begin{prop}[cf.~{\cite[Prop.~2.3]{3hamcycle}}]
\label{prop:robust_alt} Let $ \frac{1}{2}<d\le1$ and $0 < \mu < d-\frac{1}{2}$. For every 3-graph $H$ on sufficiently many vertices $n$ with $\delta_1(H) \geq d\cdot\frac{n^2}{2}$, there exist $\beta>0$ and an odd integer $\ell\geq 3$ such that the following holds. For every $v \in V(H)$, there exists an induced subgraph $R_v\subseteq L_v$ satisfying 
\begin{itemize}
\item[(i)] $|V(R_v)|> \left(\frac{1}{2}+\sqrt{\frac{d-\frac{1}{2}-\mu}{2}}\right)n$,
\item[(ii)] $e_L(V(R_v),V\setminus V(R_v))
<\mu \frac{n^2}{2}$ and $e(R_v)> \left(d - \mu\right)\frac{n^2}{2}-\frac{(n-|V(R_v)|)^2}{2}$,
\item[(iii)] and $R_v$ is $(\beta,\ell)$-robust. 
\end{itemize}
\end{prop}

Before we prove~\cref{prop:robust_alt}, note that~\cref{prop:robust} follows immediately by substituting $d = \frac{3}{4} + \alpha$ and  taking, say, $\mu = \alpha/100$, so that  $ \sqrt{1/8 + \alpha/2 - \mu/2} > \sqrt{1/8}+\alpha/2$ and $\mu < \alpha/2$.
We will now prove the stronger proposition.
\begin{proof}[Proof of~\cref{prop:robust_alt}]
The proof follows the exact same approach as~\cite[Prop.~2.3]{3hamcycle}. The only major deviations in our proof are when proving our induced subgraphs $R_v$ satisfy properties (i) and (ii). Hence we will show that induced subgraphs $R_v$ exist that satisfy properties (i) and (ii) and defer property (iii) to the proof of \cite[Prop.~2.3]{3hamcycle} for the interested reader.

    Since the proof will be the exact same for every $v \in V(H)$, we fix an arbitrary $v \in V(H)$ in what follows, writing $L$ for $L_v$ and  $R$ for $R_v$. 
  Fix $t \in \mathbb{N}$ to be the largest integer for which there exists a partition $V_1 \sqcup \cdots \sqcup V_t = V(H)$ with \begin{itemize}
        \item[(a)] $|V_1| \geq \cdots \geq |V_t| \geq \mu n/4$ and 
        \item[(b)]  $\sum_{1 \leq i < j \leq t}e_L(V_i, V_j) \leq t\left(\mu n/4\right)^2$.
    \end{itemize} 
    Clearly, $t \geq 1$ as the trivial partition $V_1 = V$ satisfies (a) and (b). 
    From (a) we infer that $t \leq 4/\mu$, and hence using (b) we have $$\sum_{1 \leq i < j \leq t} e_L(V_i, V_j) < \frac{\mu n^2}{4}.$$ 
    Let $\eta \in (0, 1]$ such that $$|V_1| = \eta n.$$ If $\eta \leq 1/2$, then we have that \begin{align*}
        \delta_1(H) \leq e(L) = \sum_{i=1}^t e_L(V_i) + \sum_{1 \leq i < j \leq t} e_L(V_i, V_j)  
        & <  \sum_{i=1}^t \frac{|V_i|^2}{2} + \frac{\mu n^2}{4} 
        \\& \leq \frac{n}{2}\sum_{i=1}^t\frac{|V_i|}{2} +\frac{\mu n^2}{4} 
       \\ &= \left(\frac{1+\mu}{2}  \right)\frac{n^2}{2} < d\frac{n^2}{2}, 
    \end{align*} contradicting our assumption on $\delta_1(H)$. Hence we can assume $\eta > 1/2$. Observe that $$\frac{\eta^2n^2}{2} \geq e_L(V_1) \ge e(L) - \frac{(n - |V_1|)^2}{2} - \frac{\mu n^2}{4} \geq \left(d - (1 - \eta)^2 - \tfrac{\mu}{2}\right)\frac{n^2}{2}.$$
    Rearranging, we have 
    \[ \eta^2 - \eta -\tfrac{1}{2}\left(d-1-\tfrac{\mu}{2}\right) \ge 0\]
    and so
    $$\Bigg(\eta - \frac{1}{2}\left(1+\sqrt{2d-1-\mu}\right)\Bigg)\Bigg(\eta   - \frac{1}{2}\left(1-\sqrt{2d-1-\mu}\right)\Bigg) \geq 0.$$ 
 Since $\eta > 1/2$, we must therefore have $\eta \ge \frac{1}{2}\left(1+\sqrt{2d-1-\mu}\right)$, and in particular \[|V_1| = \eta n \geq \left(1+\sqrt{2d-1-\mu}\right)\frac{n}{2}.\]
 Note also that $\mu + 1< 2d$. One can then calculate that  \[
 e_L(V_1) \ge \left(\frac{d}{2} - \frac{\mu}{4} + \frac{1}{2}\sqrt{2d - 1 - \mu}\right)\cdot\frac{n^2}{2} \ge \frac{\mu n^2}{4}.\]
 
    Let $W = \{w_1, \ldots, w_m\} \subseteq V_1$ be a maximal (ordered) subset of $V_1$ such that for every $i \in [m]$ we have $|N_L(w_i) \cap (V_1\setminus \{w_1, \ldots, w_{i-1}\})| < \mu n/4$. Since $e_L(V_1) > \frac{\mu n^2}{4}$, we have $V_1 \setminus W \neq \emptyset$.
    Moreover, $V_1\setminus W$ induces a subgraph of minimum degree at least $\mu n/4$ in $L$. 
    Set $U := V_1 \setminus W$ and $R = L[U]$. We now show that $R$ is our desired induced subgraph.  

    We claim $|W| < \mu n/4$. 
    Suppose for a contradiction that there exists a subset $W' = \{w_1, \ldots, w_{\lceil  \mu n/4 \rceil}\} \subseteq W$.
    Then we can replace $V_1$ in the partition $V_1 \sqcup \cdots \sqcup V_t = V$ by $W' \sqcup (V_1\setminus W')$ and obtain a partition into $t+1$ parts. 
    Observe that $$|V_1\setminus W'| \geq |V_1\setminus W| = |U| > \delta(R) \geq \mu n/4$$ and the ordering of the vertices in $W$ yields that $$e_L(W', V_1 \setminus W') \leq \sum_{w_i \in W'}|N_L(w_i) \cap (V_1\setminus \{w_1, \ldots, w_{i-1}\})| < \frac{\mu n}{4} |W'| \leq \left(\frac{\mu n}{4}\right)^2.$$ 
    Therefore the partition $W' \sqcup (V_1\setminus W') \sqcup V_2 \sqcup \cdots \sqcup V_t$  satisfies (a) and (b), contradicting that $t$ was chosen maximally.
    Hence $|W| < \mu n/4$.

    Properties (i) and (ii) of Proposition~\ref{prop:robust_alt} will now follow quickly. 
    Firstly, observe that since $|W| < \mu n/4$ that    \begin{align*}
       |U| = |V_1\setminus W| & = |V_1| - |W| \\ & > |V_1| - \frac{\mu}{4}n = \left(\eta - \frac{\mu}{4}\right)n > \left(\frac{1}{2}+\sqrt{\frac{2d-1-2\mu}{4}} \right)n,
    \end{align*} where the last inequality follows from $\eta \geq \frac{1}{2}(1 + \sqrt{2d - 1 - \mu})$. Hence $R$ satisfies Proposition~\ref{prop:robust_alt}(i). Secondly, by definition of $W$ and since $t \leq 4/\mu$ and $|W| < \mu n/4$, we have
    \begin{align*}
        e_L(U, V\setminus U) & = \sum_{i=2}^t e_L(U, V_i) + e_L(U, W) \\ & \leq \sum_{i=2}^t e_L(V_1, V_i) + \frac{\mu n}{4}  |W| 
        \\&< t\left(\frac{\mu n}{4}\right)^2 + \left(\frac{\mu n}{4}\right)^2 \le \left(\frac{\mu}{2}+ \frac{\mu^2}{8}\right)\frac{n^2}{2}.
    \end{align*} Since $\mu < d-\frac{1}{2} < 1$, this ensures $R$ satisfies the first inequality of Proposition~\ref{prop:robust_alt}(ii).
    The second inequality of Proposition~\ref{prop:robust_alt}(ii) is seen as follows \begin{align*} e(R) = e(L) - e_L(U, V\setminus U) - e_L(V\setminus U) & 
    \geq d\frac{n^2}{2} - \mu \frac{n^2}{2} - \frac{(n - |U|)^2}{2}.\end{align*}
\end{proof}

\end{document}